\documentclass[a4paper,12pt]{amsart}
\usepackage[a4paper,marginratio={1:1},scale={0.72,0.74},footskip=7mm,headsep=10mm]{geometry}

\usepackage{amsfonts}
\usepackage{amssymb,amsmath,latexsym}
\usepackage[dvips]{graphics}
\usepackage{graphicx,subfigure}
\usepackage[dvips]{color}
\usepackage[T1]{fontenc}
\usepackage[active]{srcltx}
\usepackage{amsmath}
\usepackage{amsfonts}
\usepackage{amssymb}
\usepackage{psfrag}
\usepackage{color}
\usepackage{url}
\usepackage{amsthm}
\usepackage{array}
\usepackage{pst-tree}
\usepackage{lscape}
\usepackage{dsfont}
\usepackage[T1]{fontenc}
\usepackage{pstricks,pstricks-add}\usepackage{mathrsfs}  

\newtheorem{theorem}{Theorem}[section]
\newtheorem{proposition}{Proposition}[section]
\newtheorem{remark}{Remark}[section]

\newtheorem{corollary}{Corollary}[section]

\newtheorem{lemma}{Lemma}[section]

\def\e{\mathbb{E}}
\def\p{\mathbb{P}}
\def\R{\mathbb{R}}
\newcommand{\ind}{\mbox{\rm 1\hspace{-0.04in}I}}

\title[Density behaviour related to L\'evy processes]{Density behaviour related to L\'evy processes}

\author{Lo\"i{}c Chaumont and Jacek Ma{\l}ecki}
\address{Lo\"i{}c Chaumont, \\ LAREMA, D\'e{}partement de Math\'e{}matique\\ Universit\'e d'Angers \\ Bd Lavoisier - 49045, 
Angers Cedex 01, France}
\email{loic.chaumont@univ-angers.fr}
\address{Jacek Ma{\l}ecki\\ Wydzia{\l} Matematyki \\ Politechnika Wroc{\l}awska \\ ul. Wybrze{\.z}e Wyspia{\'n}\-skiego 27\\ 
50-370 Wroc{\l}aw, Poland}
\email{jacek.malecki@pwr.edu.pl}



\date{\today}

\thanks{J. Ma{\l}ecki is supported by the Polish National Science Centre (NCN) grant no. 2015/19/B/ST1/01457.}

\begin{document}

\begin{abstract} Let $p_t(x)$, $f_t(x)$ and $q_t^*(x)$ be the densities at time $t$ of a real L\'evy process,
its running supremum and the entrance law of the reflected excursions at the infimum. We provide relationships 
between the asymptotic behaviour of $p_t(x)$, $f_t(x)$ and $q_t^*(x)$, when $t$ is small and $x$ is large.
Then for large $x$, these asymptotic behaviours are compared to this of the density of the L\'evy measure.
We show in particular that, under mild conditions, if $p_t(x)$ is comparable to $t\nu(x)$, as $t\rightarrow0$ and
$x\rightarrow\infty$, then so is $f_t(x)$.
\end{abstract}

\maketitle

\section{Introduction}

Real L\'evy processes are very often involved in stochastic modeling of time dependent dynamics. 
For such a process $(X,\p)$, it is often crucial to have information on the law of $X_t$ and this of its past supremum 
at time $t>0$, $\overline{X}_t=\sup_{s\le t}X_s$, in terms of the characteristics of $(X,\p)$. However, most of the time there 
is no explicit form for these distributions. The only existing results regard the asymptotic behavior of the distribution 
functions $\p(X_t\le x)$ and $\p(\overline{X}_t\le x)$, as $t\rightarrow\infty$ and $x\rightarrow0$ or $\p(X_t>x)$ and 
$\p(\overline{X}_t>x)$, as $t\rightarrow0$ and $x\rightarrow\infty$, and very few works go beyond the stable case, 
see \cite{bib:kmr13}. 

In the present paper, assuming absolute continuity of the process $(X,\p)$ and its L\'evy measure on $[0,\infty)$, we 
obtain relationships between the asymptotic behaviour of the density of $X_t$ and this of $\overline{X}_t$, for small
times and large space values. More specifically, let $p_t(x)$, $f_t(x)$ and $\nu(x)$ be the densities on $[0,\infty)$ of $X_t$, 
$\overline{X}_t$ and the L\'evy measure, respectively. Then under mild conditions essentially bearing on the smoothness 
of $\nu(x)$, we show in a series of results that if the asymptotic behaviour of $p_t(x)$, as $t\rightarrow0$ and 
$x\rightarrow\infty$ is given by $t\nu(x)$, then this is also the case for $f_t(x)$. The most famous example where this 
situation holds is the  stable case  \cite{bib:ds10}, \cite{ku1}, \cite{ku2} but other examples can be 
constructed in the spectrally positive case or for subordinated Brownian motions as it is done here. Actually, in all the 
particular situations where the asymptotic behaviour of $p_t(x)$ can be described, it is as expected and the behaviour 
of $f_t(x)$ can be derived. 

Another important density in our work is this of the entrance law of the excursions of $X$ reflected at its past infimum. 
Its close connection with $p_t(x)$ and $f_t(x)$ allows us to use results from fluctuation theory and is the main argument of 
our proofs. Whenever the half-line $(0,\infty)$ is regular, this density, denoted here by $q_t^*(x)$, corresponds to the density 
of the renewal measure of the space-time upward ladder process. This relationship reinforces the interest of the asymptotic 
behaviour of $q_t^*(x)$ as $t\rightarrow0$ and $x\rightarrow\infty$, as described here. 

The next section is devoted to reminders on the essentials of fluctuation theory and some preliminary results regarding 
the renewal measure of the ladder process. Then we state our main results in Section 3, first with asymptotic results for 
$t\to0^+$, then we give estimates in small $t$ and large $x$ regime and finally we state asymptotic results for $x\to \infty$. 
In Section \ref{sec:examples} we give some examples in order to illustrate the possible applications of our results.
Proofs will be given in Section \ref{proofs}.

\section{Notation and preliminary results}

We denote by $\mathcal{D}$  the space of c\`{a}dl\`{a}g paths $\omega:[0,\infty )
\rightarrow\mathbb{R}\cup\{\infty\}$ with lifetime $\zeta(\omega )=\inf \{t\ge0:\omega _{t}=\infty\}$, with the usual convention that 
$\inf\emptyset=+\infty$. The space $\mathcal{D}$ is equipped with the Skorokhod topology,  its Borel $\sigma $-field $\mathcal{F}$
and the usual completed filtration $(\mathcal{F}_{s},s\geq 0)$ generated by the coordinate process $X=(X_{t},t\geq 0)$ on the space 
$\mathcal{D}$. We write  $\underline{X}$ and $\overline{X}$ for the infimum and supremum processes, that is
\[\underline{X}_{t}=\inf \{X_{s}:0\leq s\leq t\}\;\;\;\mbox{and}\;\;\;
\overline{X}_{t}=\sup \{X_{s}:0\leq s\leq t\}\,.\]
We also define the first passage time by $X$ in the open half line  $(-\infty,0)$ by:
\[\tau_0^-=\inf\{t>0:X_t<0\}\,.\]
We denote by $\mathbb{P}_x$ the law on $(\mathcal{D},\mathcal{F})$ of a L\'{e}vy process starting from $x\in\mathbb{R}$ and
we will set $\mathbb{P}:=\mathbb{P}_0$. We will only mention $X$ which will be understood as a L\'evy process under the family of 
probability measures $(\p_x)_{x\in\mathbb{R}}$. 
\medskip

\noindent {\it In all this paper, we exclude the case where $|X|$ is a subordinator and this of compound 
Poisson processes.}

\medskip

Recall that both reflected processes $X-\underline{X}$ and $\overline{X}-X$ are Markovian. Moreover the state 0 is regular for 
itself, for $X-\underline{X}$ if and only if 0 is regular for $(-\infty,0)$, for the process $X$. When it is the case, we will simply write 
that $(-\infty,0)$ is regular. Similarly, we will write that $(0,\infty)$ is regular when 0 is regular for $(0,\infty)$, for the process $X$,
which is equivalent to the fact that 0 is regular for itself, for the process $\overline{X}-X$. Since $X$ is not a compound 
Poisson process, at least one of the half lines $(-\infty,0)$ or $(0,\infty)$ is regular.\\

If $(-\infty,0)$ is regular, then the local time at 0 of the process $X-\underline{X}$ is a continuous, increasing, additive 
functional which we will denote by $L^*$, satisfying $L_0^*=0$, a.s., and such that the support of the measure $dL_t^*$ is the set
$\overline{\{t:X_t=\underline{X}_t\}}$. Moreover $L^*$ is unique up to a multiplicative constant. We will normalize it by
\begin{equation}\label{norm1}
\e\left(\int_0^\infty e^{-t}\,dL_t^*\right)=1\,.
\end{equation}

\medskip

When  $(-\infty,0)$ is not regular, the set $\{t:(X-\underline{X})_t=0\}$ is discrete and following
\cite{bib:b96} and \cite{ky}, we define the local time $L^*$ of $X-\underline{X}$ at 0 by
\begin{equation}\label{norm2}
L_t^*=\sum_{k=0}^{R_t}{\rm\bf e}^{(k)}\,,
\end{equation}
where for $t>0$, $R_t=\mbox{Card}\{s\in(0,t]:X_s=\underline{X}_s\}$, $R_0=0$ and ${\rm\bf e}^{(k)}$, $k=0,1,\dots$ is a 
sequence of independent and exponentially distributed random variables with parameter
\begin{equation}\label{alpha}
\gamma=\left(1-\e(e^{-\tau^-_0})\right)^{-1}\,.
\end{equation}

The It\^o measure $n^*$ of the excursions away from 0 of the process $X-\underline{X}$ is characterized by the
{\it compensation formula}. More specifically, it is defined for any positive and predictable process $F$ by,
\begin{equation}\label{compensation}\e\left(\sum_{s\in G}F(s,\omega,\epsilon^s)\right)=
\e\left(\int_0^\infty dL_s^*\left(\int_E F(s,\omega,\epsilon)n^*(d\epsilon)\right)\right)\,,
\end{equation}
where $E$ is the set of excursions, $G$ is the set of left end points of the excursions, and $\epsilon^s$ is the excursion  
which starts at $s\in G$.
We refer to \cite{bib:b96}, Chap.~IV, \cite {ky}, Chap.~6 and \cite{do} for more detailed definitions and some
constructions of $L^*$ and $n^*$. The measure $n^*$ is $\sigma$-finite and it is finite if and only if $(-\infty,0)$ is not regular.\\

When $(-\infty,0)$ is not regular, the measure $n^*$ of the excursions away from 0  is proportional to
the distribution of the process $X$ under the law $\mathbb{P}$, killed at its first passage
time in the negative half line. More formally, the excursion that starts at $0\in G$ is given by  
$\epsilon^{0}=(X_t\ind_{\{t<\tau_0^-\}}+\infty\cdot\ind_{\{t\ge\tau_0^-\}})$
and for any bounded Borel functional $K$ on $E$,
\begin{equation}\label{excdisc}
\int_EK(\epsilon)n^*(d\epsilon)=\gamma\,\e[K(\epsilon^{0})]\,.
\end{equation}
In this case, the normalization (\ref{norm1}) and the compensation formula (\ref{compensation}) are direct consequences 
of definitions (\ref{norm2}), (\ref{excdisc}) and the strong Markov property.

\medskip

In any case, $n^*$ is a Markovian measure whose semigroup is this of the killed L\'evy process when it enters in the negative 
half line. More specifically, for $x>0$, let us denote by $\mathbb{Q}_{x}^*$ the law of the process
$(X_t\ind_{\{t<\tau_0^-\}}+\infty\cdot\ind_{\{t\ge\tau_0^-\}},\,t\ge0)$ under $\p_x$,
that is for $\Lambda \in \mathcal{F}_{t}$,
\begin{equation}
\label{4524}
\mathbb{Q}_{x}^*(\Lambda ,t<\zeta )=\mathbb{P}_{x}(\Lambda ,\,t<\tau_{0}^-)\,.
\end{equation}
Then for all Borel positive functions $f$ and $g$ and for all $s,t>0$,
\begin{equation}
\label{4527}
n^*(f(X_t)g(X_{s+t}),s+t<\zeta)=n^*(f(X_t)\e^{\mathbb{Q}^*}_{X_t}(g(X_s)),s<\zeta)\,,
\end{equation}
where $\e_x^{\mathbb{Q}^*}$ means the expectation under $\mathbb{Q}_x^*$. Recall that $\mathbb{Q}_{0}^*$ is well defined 
when $(-\infty,0)$ is not regular and in this case, from (\ref{excdisc}), we have $\mathbb{Q}_{0}^*=\gamma^{-1}n^*$.\\  

We write $X^*:=-X$ for the dual L\'evy process whose law under $\p_x$ will be denoted by $\p_x^*$, that is 
$(X^*,\p_x)=(X,\p_x^*)$ and we set $\p^*=\p^*_0$. Then we define the probability measures $\mathbb{Q}_x$ in the same way 
as in (\ref{4524}) with respect to the dual process $(X,\p^*)$. Let us denote by $q_{t}^*(x,dy)$  (resp. $q_{t}(x,dy)$)  the 
semigroup of the strong Markov process $(X,\mathbb{Q}_x^*)$ (resp. $(X,\mathbb{Q}_x)$) and by $q_t^*(dx)$, 
$t>0$, the entrance law of $n^*$, that is for any positive Borel function $f$,
\begin{equation}\label{2682}
\int_{[0,\infty)}f(x)\,q_t^*(dx)=n^*\left(f(X_t),\,t<\zeta\right)\,.
\end{equation}
The local time at 0 of the reflected process at its supremum $\overline{X}-X=X^*-\underline{X}^*$ and the measure of its 
excursions away from 0 are defined in the same way as for $X-\underline{X}$. They are respectively denoted by $L$ and 
$n$. Then the entrance law $q_t(dx)$ of $n$ is defined in the same way as $q_t^*(dx)$. 

\medskip

Let us now recall the definition of the ladder processes. The ladder time processes $\tau$ and $\tau^*$, and the ladder height 
processes $H$ and $H^*$ are the following (possibly killed) subordinators:
\[\tau_t=\inf\{s:L_s>t\}\,,\;\;\tau^*_t=\inf\{s:L_s^*>t\}\,,\;\;H_t=X_{\tau_t}\,,\;\;H^*_t=-X_{\tau_t^*}\,,\;\;t\ge0\,,\]
where $\tau_t=H_t=+\infty$, for $t\ge\zeta(\tau)=\zeta(H)$ and $\tau_t^*=H_t^*=+\infty$, for $t\ge\zeta(\tau^*)=\zeta(H^*)$. 
Recall that the bivariate processes $(\tau,H)$ and $(\tau^*,H^*)$ are themselves L\'evy processes. We denote by $\kappa$ 
and $\kappa^*$ their respective Laplace exponents, that is for all $u,\alpha,\beta\ge0$,
\begin{equation}\label{7342}
\e(e^{-\alpha\tau_u-\beta H_u})=e^{-u\kappa(\alpha,\beta)}\;\;\;\mbox{and}\;\;\;\e(e^{-\alpha\tau^*_u-\beta H^*_u})=
e^{-u\kappa^*(\alpha,\beta)}\,.
\end{equation}
The drifts  ${\tt d}$ and ${\tt d}^*$  of the subordinators $\tau$ and $\tau^*$ satisfy
\begin{equation}\label{delta}
\int_0^t\ind_{\{X_s=\overline{X}_s\}}\,ds={\tt d}L_t\,,\;\;\;\int_0^t\ind_{\{X_s=\underline{X}_s\}}\,ds={\tt d}^*L_t^*.
\end{equation}
Recall that ${\tt d}>0$ if and only if $(-\infty,0)$ is not regular.  In this case, we can check that ${\tt d}=\gamma^{-1}$, see \cite{bib:ch13}. 
Similarly, ${\tt d}^*>0$ if and only if 0 is not regular for $(0,\infty)$, so that ${\tt d}{\tt d}^*=0$ always holds since 0 is necessarily 
regular for at least one of the half lines. 

\medskip

Let us write 
\[g(dt,dx)=\int_0^\infty\p(\tau_u\in dt,\,H_u\in dx)\,du\]
for the renewal measure of the ladder process $(\tau,H)$. 
Then $g(dt,dx)$ satisfies two useful convolution equations that are given in the following lemma. Note that 
(\ref{2681}) is the continuous time counterpart of (10) in \cite{ad}. It was stated and proved in the preprint version of \cite{ac}. 
Moreover, (\ref{2294}) is stated as (4.9) in \cite{ac} under the assumption that 0 is regular for both $(-\infty,0)$ and $(0,\infty)$ 
but no proof is given. A proof of Theorem \ref{7551} is given in Section \ref{subsec:Aux}.

\begin{theorem}\label{7551}
\label{thm:qtpt:equalities}
Let $X$ be any real L\'evy process which is not a compound Poisson process
and such that $|X|$ is not a subordinator. Then the renewal measure $g(dt,dx)$ satisfies the following identities between 
measures on $[0,\infty)^2$,
\begin{equation}\label{2681}
tg(dt,dx)=\int_{s=0}^t\int_{z=0}^x\p_z(X_{t-s}\in dx)g(ds,dz)\,dt
\end{equation}
and
\begin{equation}\label{2294}
xg(dt,dx)=\int_{s=0}^t\int_{z=0}^x\frac{x-z}{t-s}\p_z(X_{t-s}\in dx)g(ds,dz)\,dt\/.
\end{equation}
\end{theorem}

\noindent The entrance law $q_t^*(dx)$ is closely related to the renewal measure $g(dt,dx)$ through the
following identity    
\begin{equation}\label{2598}
{\rm d}^*\delta_{\{(0,0)\}}(dt,dx)+q_t^*(dx)\,dt=g(dt,dx),
\end{equation}
which can be found in Lemma 1 of \cite{bib:ch13}. Since $X$ is not a 
compound Poisson process, the only possible atom for the law of $\tau_u$ and $H_u$ is 0 and this is the case if and only if 
$(0,\infty)$ is not regular. Moreover, $\tau_u=0$ if and only if $H_u=0$. It follows from these observations that $g(\{0\},\{0\})>0$ 
if and only if $(0,\infty)$ is not regular and more precisely $g(\{0\},\{0\})={\tt d}^*$. Moreover in general, $g(\{a\},A)=g(A,\{a\})=0$ 
whenever $a>0$ or $A\subset(0,\infty)$. We derive from these remarks and (\ref{2598}) that the measure $q_t^*(dx)$ on $[0,\infty)$
is atomless for all $t>0$. Moreover, plugging relation(\ref{2598}) into (\ref{2681}) and (\ref{2294}), we obtain the following identity 
between measures on $(0,\infty)^2$,
\begin{equation}
tq_t^*(dx)\,dt={\rm d}^*\p(X_t\in dx)\,dt+\int_{s=0}^t\int_{z=0}^x\p_z(X_{t-s}\in dx)q_s^*(dz)ds\,dt\label{1732}
\end{equation}
and
\begin{equation}
xq_t^*(dx)\,dt={\rm d}^*\frac{x}{t}\p(X_t\in dx)\,dt+\int_{s=0}^t\int_{z=0}^x\frac{x-z}{t-s}\p_z(X_{t-s}\in dx)q_s^*(dz)ds\,dt.
\label{5144}
\end{equation} 
\medskip

Let us now focus on the links between these different objects and the law of the supremum.
From  Corollary 3 of \cite{bib:ch13}, the law of $\overline{X}_t$ fulfills the following representation
\begin{equation}
   \label{eq:sup:formula}
  \p(\overline{X}_t\in dx) = \int_0^t n(t-u<\zeta)q_u^*(dx)du+{\tt d}q_t^*(dx)+{\tt d^*}n(t<\zeta)\delta_{\{0\}}(dx)\/.
\end{equation}
In all this paper, by absolute continuity on some set, we mean absolute continuity with respect to the  
Lebesgue measure on this set.
From Lemma 1 of \cite{bib:ch13}, the distribution $\p(X_t\in dx)$ is absolutely continuous on $[0,\infty)$ if and only if 
$q_t^*(dx)$ is absolutely continuous on $[0,\infty)$ and when this is the case, the law of $\overline{X}_t$ is absolutely continuous 
on $(0,\infty)$. We will denote by $f_t(x)$ the corresponding density function. Moreover, this law has an atom at $0$ if and only
if $(0,\infty)$ is not regular. The following representation derived from (\ref{eq:sup:formula}) will frequently be used in the sequel.
\begin{equation}
  \label{eq:ft:formula}
  f_t(x) = \int_0^t n(t-u<\zeta)q_u^*(x)du+{\tt d}q_t^*(x)\/,\quad x>0\/.
\end{equation}
Moreover, by Corollary 4 of \cite{bib:ch13}, the following integral representation of the density of $X_t$ under $\p$ holds
\begin{eqnarray}
  \label{eq:pt:formula:exp}
	  p_t(x) &=& \int_0^t \int_0^\infty q_u^*(x+z)q_{t-u}(dz)du+{\tt d}q_t^*(x)\/,\quad x>0\/.
\end{eqnarray}
Finally let us write expression (\ref{5144}) in the absolute continuous case
\begin{equation}
\label{eq:ACformula}
xq_t^*(x) = {\tt d^*}\dfrac{xp_t(x)}{t}+\int_0^t du\int_0^xdz\,q_u^*(z)\frac{x-z}{t-u}p_{t-u}(x-z)\/,\quad x,t>0\/.
\end{equation}
We emphasize that from (\ref{2598}) when the distribution $\p(X_t\in dx)$ is absolutely continuous and $(0,\infty)$ is regular, 
i.e.~$d^*=0$, the renewal measure $g(dt,dx)$ is absolutely continuous on $[0,\infty)^2$ and the function $(t,x)\mapsto q_t^*(x)$
corresponds to its density. 

\section{Main results}

We recall once again that in all the results of this paper the process $|X|$ is not a subordinator and that $X$ is not a compound 
Poisson process. 

We shall denote by $\nu(dx)$ the L\'evy measure of $X$. Apart from Proposition \ref{prop:qt:tail}, most of our 
results relate to densities of the Lévy measure, the process at fixed times, the entrance law of the reflected  excursions 
and the past supremum, hence the minimal assumptions for the statements of our results to make sense are as follows:

\medskip

\[(H)\;\;\;\;\;\;
\begin{minipage}{5.5in}
{\it  For all $t>0$, the law of $X_t$ under $\p$ is absolutely continuous on $[0,\infty)$ and the L\'evy measure $\nu(dx)$ is 
absolutely continuous on $(0,\infty)$.}
\end{minipage}\]

\medskip

\noindent Therefore, unless explicitly mentioned in Proposition \ref{prop:qt:tail}, assumption $(H)$ will be in force all along this 
paper. Then recall from the previous section that under assumption $(H)$, $q_t^*(dx)$ is absolutely continuous on $[0,\infty)$ 
and $\p(\overline{X}_t\in dx)$ is absolutely continuous on $(0,\infty)$. The corresponding densities will be  denoted as follows:
\[\p(X_t\in dx)=p_t(x)\,dx\;\;;\;\;\nu(dx)=\nu(x)\,dx\;\;;\;\;q_t^*(dx)=q_t^*(x)\,dx\;\;;\;\;\p(\overline{X}_t\in dx)=f_t(x)\,dx.\]

\subsection{Asymptotic results for $t\to0^+$} Let us start by looking at the small time asymptotic behaviour of $p_t(x)$, $q^*_t(x)$ and
$f_t(x)$. Recall that the general result states that for any L\'evy process the measure $\p(X_t\in dx)/t$ tends weakly as $t\to 0^+$ to the 
L\'evy measure $\nu(dx)$ on $|x|>\varepsilon$ with fixed $\varepsilon>0$, see Exercise 1, Chap.~I in \cite{bib:b96}. Obviously, it does 
not imply in general that $p_t(x)/t$ tends to $\nu(x)$ for large $x$, but one may expect such a behaviour of $p_t(x)$ to hold and, in fact, 
it is true in many cases. Some of them are presented in Section \ref{sec:examples}.
In the first result we show the relation between the asymptotic behaviour of $p_t(x)$ and this of $q_t^*(x)$ as 
$t\to 0^+$, for processes whose density of the L\'evy measure is uniformly continuous (for large $x$). Note that $\nu(x)$ is uniformly 
continuous for example whenever it is continuous (for large $x$) and tends to $0$ at infinity. Let us introduce the following notation
\begin{equation}
  \label{eq:NN*:defn}
	N^*(t) = \int_0^t n^*(s<\zeta)ds\/.
\end{equation}


\begin{theorem}
\label{thm:asympt:t}
Assume that $\nu(x)$ is uniformly continuous on $(x_1,\infty)$ for some $x_1>0$. Then the following assertions are equivalent.
\begin{itemize}
\item[(a)] There exists $x_0>0$ such that
\begin{equation}
  \label{eq:ptnu:asympt:tweak}
	\lim_{t\to 0^+}\frac{p_t(x)}{t}=\nu(x)
\end{equation}
uniformly for $x\geq x_0$.
\item[(b)] There exists $x_0>0$ such that 
\begin{equation}
  \label{eq:qtnu:asympt:tweak}
	\lim_{t\to 0^+}\frac{q_t^*(x)}{{\tt d^*}+N^*(t)}=\nu(x)
\end{equation}	
	uniformly for $x\geq x_0$.
	\end{itemize}
\end{theorem}

This result will be used to establish the following relation between the asymptotic behaviour as $t$ goes to $0$ of $f_t(x)$ and 
this of $p_t(x)$. 

\begin{theorem}
\label{thm:ftx:tzero}
Assume that $\nu(x)$ is uniformly continuous on $(x_1,\infty)$ for some $x_1>0$. If there exists $x_0>0$ such that
\begin{equation}
  \label{eq:ptnu:asympt:tweak2}
	\lim_{t\to 0^+}\frac{p_t(x)}{t}=\nu(x)
\end{equation}
uniformly for $x\geq x_0$, then there exists $x_2>0$ such that
\begin{equation}
  \label{eq:ftnu:asympt:tweak}
	\lim_{t\to 0^+}\frac{f_t(x)}{t}=\nu(x)
\end{equation}
uniformly for $x\geq x_2$.
\end{theorem}
As an immediate consequence of Theorems \ref{thm:asympt:t} and \ref{thm:ftx:tzero}, we obtain that whenever $p_t(x)/t$ 
tends uniformly on $(\varepsilon,\infty)$ to a uniformly continuous L\'evy density $\nu(x)$, then ${q_t^*(x)dx}/{{\tt d^*}+N^*(t)}$ and 
$f_t(x)dx/t$ tend weakly on $\{x>\varepsilon\}$ to $\nu(x)dx$ as $t\to0^+$. However, it seems that these weak convergence results 
are much more general and do not even require absolute continuity. These observations lead us to the following conjecture.\\

\noindent{\bf Conjecture}{\it $\;\;$
   For any L\'evy process $(X,\p)$ and every $\varepsilon>0$,
	  $$
		  \dfrac{\p(\overline{X}_t\in dx)}{t} \Rightarrow \nu(dx)\/,\quad\quad  \dfrac{q_t^*(dx)}{{\tt d^*}+N^*(t)} \Rightarrow \nu(dx)
		$$
		on $\{x>\varepsilon\}$, as $t\to0^+$.}

\medskip

We finish this series of results with the observation that the result from Theorem \ref{thm:asympt:t} can also be applied to determine 
the behaviour of the joint law $(X_t,\overline{X}_t)$. Note that in this case, we also require the regularity of $p_t(dx)$ and $\nu(dx)$ 
on both positive and negative half-lines. 
\begin{theorem}
\label{thm:asympt:t:joint}
Assume that $p_t(dx)=p_t(x)dx$ and $\nu(dx)=\nu(x)dx$, on $\R$ and $\R\setminus\{0\}$ respectively, with $\nu(x)$ being uniformly 
continuous on $\{|x|\geq \varepsilon\}$, $\varepsilon>0$. Then the joint law $\p(X_t\in dx,\overline{X}_t\in dy)$ is absolutely continuous 
on $\R\times(0,\infty)$ with density $h_t(x,y)$. If $p_t(x)/t$ tends to $\nu(x)$, as $t\to0^+$, uniformly on $x>x_2$, then there exists 
$x_1>0$ such that 
\begin{equation}
   \lim_{t\to 0^+}\dfrac{h_t(x,y)}{t} = \nu(y-x)\nu(y),
\end{equation} 
uniformly on $x,y\geq x_1$.

\end{theorem}
 
\begin{remark}
Also in this case it seems reasonable to expect that the corresponding weak convergence result holds in full generality, i.e. 
for every L\'evy process $(X,\p)$ and $\varepsilon>0$,
$$
\frac{1}{t}\p(X_t\in dx,\overline{X}_t-X_t\in dy) \Rightarrow \nu(dx)\nu(dy)\/,\quad t\to 0^+\/,
$$
on $\{x>\varepsilon, y>\varepsilon\}$. Note that this result can be deduced from the Conjecture above together with the well-known 
fact that if $e$ is an exponentially distributed random variable which is independent of $X$, then the random variables 
$\overline{X}_e-X_e$ and $\overline{X}_{e}$ are independent. Moreover $\overline{X}_e-X_e$ has the same law as 
$\overline{X^*_e}$.
\end{remark}

\subsection{Estimates in small $t$ and large $x$ regime} 

Since the uniform convergence of $p_t(x)/t$ to $\nu(x)$ implies that $p_t(x)\approx t\nu(x)$ for $x$ large and $t$ small, it is natural to 
ask how those estimates are related to the corresponding estimates of $q_t^*(dx)$ and $f_t(x)$. The following results are intended to 
give answers to these questions. For this purpose, we first state some auxiliary results related to general bounds of the entrance law. 
They are crucial in the proofs, but since we prove them in a fairly large generality, they also have an interest of their own. 
That is why they are stated here as separate results. Recall that in the following Proposition we do not require absolute continuity.

\begin{proposition}
\label{prop:qt:tail}
For every $\varepsilon>0$ and $x_0>0$ there exists $t_0>0$ such that 
\begin{equation}
\label{eq:qt:tail}
   q_t^*((x,\infty))\leq \varepsilon \,\frac{n^*(t<\zeta)}{h^*(x)}\/,\quad \quad  q_t((x,\infty))\leq \varepsilon \,\frac{n(t<\zeta)}{h(x)}\/,
\end{equation}
for every $0<t<t_0$ and $x>x_0$. If ${\tt d}=0$ $($resp.~${\tt d^*}=0)$, then we can take $x_0=0$ in the estimates of 
$q_t^*((x,\infty))$ $($resp.~$q_t((x,\infty)))$. Moreover, for given $\varepsilon>0$ and $t_0>0$, there exists $x_0>0$ such that 
\begin{equation}
\label{eq:qt:tail2}
   q_t^*((x,\infty))\leq \varepsilon \,\frac{n^*(t<\zeta)}{h^*(x)}\/,\quad \quad q_t((x,\infty))\leq \varepsilon \,\frac{n(t<\zeta)}{h(x)}\/,
\end{equation}
for every $x>x_0$ and $t<t_0$.
\end{proposition}

\noindent For all the remaining results we will assume that hypothesis $(H)$ on absolute continuity is in force.

\begin{proposition}
\label{prop:qt:general}
Assume that there exist $x_0, t_0>0$ such that $\nu(x)$ and $p_t(x)/t$ are bounded for $x>x_0$ and $t<t_0$. Then there exists 
$x_1>0$ and $c>0$ such that
\begin{equation}
   \label{eq:qt:general}
	  \frac{q_t^*(x)}{{\tt d^*}+N^*(t)}\leq c\/,
\end{equation}
for every $x>x_1$ and $t<t_0$. Conversely, if $(\ref{eq:qt:general})$ holds, then $p_t(x)/t$ is bounded for large $x$ and small $t$.
\end{proposition}	

In the next results, we establish relations between sharp estimates of $p_t(x)$ and $q_t^*(x)$. The regularity condition 
\eqref{eq:nu:monoton} below is not very restrictive and it holds in particular, when $\nu(x)$ is non-increasing or it is comparable 
with a non-increasing functions (for large $x$). Condition \eqref{eq:nu:double} holds in particular when $\nu(x)$ is
regularly or slowly varying at infinity. We write $f(t,x)\stackrel{x_0,t_0}{\approx}g(t,x)$ if there exists a constant 
$c=c(x_0,t_0)>0$ (possibly depending on $x_0$, $t_0$) such that $c^{-1}\,f(t,x)\leq g(t,x) \leq c\,f(t,x)$, for every $x$ and $t$ 
belonging to some domain depending $x_0$ and $t_0$. The notations $f(t,x)\stackrel{x_0,t_0}{\leq}g(t,x)$ and 
$f(t,x)\stackrel{x_0,t_0}{\geq}g(t,x)$ mean existence of a constant $c=c(x_0,t_0)>0$ such that $f(t,x)\leq c\,g(t,x)$ and 
$f(t,x)\geq c\,g(t,x)$ hold respectively, for every $x$ and $t$ belonging to some domain depending $x_0$ and $t_0$. 
		
\begin{theorem}
	\label{thm:qtx:bounds}
	  Assume that for every $r>0$ there exists $c=c(r)>0$ such that 
	   \begin{eqnarray}
	  \label{eq:nu:monoton}
	  \nu(y)\leq c(r)\nu(x)\/,
	\end{eqnarray}
	for every $y>x>r$, and
	\begin{eqnarray}
	 \label{eq:nu:double}
	 \nu(x)\leq c(r)\,\nu(2x)\/,\quad \textrm{ whenever }x>r\/.
	\end{eqnarray}
Then the following conditions are equivalent
		\begin{itemize}
    \item[$(i)$] For every $x_0,t_0>0$, 
		\begin{equation}
	  \label{eq:ptnu:bounds}
		p_t(x) \stackrel{x_0,t_0}{\approx} t\nu(x),
	  \end{equation}
		for every $t<t_0$ and $x>x_0$.
		\item[$(ii)$] For every $x_0,t_0>0$, 
		\begin{eqnarray}
		  \label{eq:qtx:bounds}
		   q_t^*(x) \stackrel{x_0,t_0}{\approx}\nu(x)({\tt d^*}+N^*(t)),
		\end{eqnarray}
		for every $t<t_0$ and $x>x_0$.
		\end{itemize}
\end{theorem}

\begin{remark}
 \label{rem:thmconditions}
  Note that not all the regularity of the L\'evy measure density $\nu(x)$ is necessary to show upper or lower estimates separately. 
  More precisely, the condition \eqref{eq:nu:monoton} is sufficient to show that the lower bounds in \eqref{eq:ptnu:bounds} implies 
  the lower bounds in \eqref{eq:qtx:bounds} and  the upper bounds in \eqref{eq:qtx:bounds} implies the upper bounds in 
  \eqref{eq:ptnu:bounds}. The condition \eqref{eq:nu:double} is sufficient to prove that the lower bounds in \eqref{eq:qtx:bounds} 
  implies the lower bounds in \eqref{eq:ptnu:bounds}. Finally, both conditions \eqref{eq:nu:monoton} and \eqref{eq:ptnu:bounds} 
  are required to show that the upper bounds in \eqref{eq:ptnu:bounds} implies the upper bounds in \eqref{eq:qtx:bounds}.
\end{remark}

\begin{remark}
   The statement of Theorem $\ref{thm:qtx:bounds}$ holds true if we replace "for every $x_0,t_0>0$" by 
   "there exists $x_0,t_0>0$" in conditions $(i)$ and $(ii)$ simultaneously. Remark $\ref{rem:thmconditions}$ can be also applied 
   in this case.  
\end{remark}

\noindent We use the results of Theorem \ref{thm:qtx:bounds} to study the corresponding estimates for a density of a supremum 
process. Theorem \ref{thm:supremum:bounds} is a justification of an intuitively obvious observation that a large level $x>0$ is 
obtained by the supremum process at given (small) time $t$ by one big jump, which is described by the L\'evy measure density 
$\nu(x)$.  

\begin{theorem}
\label{thm:supremum:bounds}
Assume that \eqref{eq:nu:monoton} holds. Then if for every $x_0,t_0>0$
$$
  p_t(x)\stackrel{x_0,t_0}{\geq} t\nu(x)\/,\quad t<t_0,x>x_0,
$$
then for every  $x_0,t_0>0$
$$
 f_t(x)\stackrel{x_0,t_0}{\geq} t\nu(x)\/,\quad t<t_0,x>x_0\/.  
$$
If additionally \eqref{eq:nu:double} holds, then if for every $x_0,t_0>0$  
$$
  p_t(x)\stackrel{x_0,t_0}{\leq} t\nu(x)\/,\quad t<t_0,x>x_0, 
$$
then for every $x_0,t_0>0$ 
$$
 f_t(x)\stackrel{x_0,t_0}{\leq} t\nu(x)\/,\quad t<t_0,x>x_0\/.  
$$
In particular, if both \eqref{eq:nu:monoton} and \eqref{eq:nu:double} hold and for every $x_0,t_0>0$ 
$$
p_t(x)\stackrel{x_0,t_0}{\approx} t\nu(x)\/,\quad t<t_0,x>x_0,
$$
then for every $x_0,t_0>0$ 
\begin{equation}
  \label{eq:supremum:comparability}
f_t(x)\stackrel{x_0,t_0}{\approx} p_t(x) \stackrel{x_0,t_0}{\approx}  t\nu(x)\/,\quad t<t_0,x>x_0 \/.
\end{equation}
\end{theorem}

\begin{remark}
We can replace "for every $x_0,t_0>0$" by "there exists $x_0,t_0>0$ such that" in each statement of the theorem above 
$($simultaneously in the corresponding hypothesis and the conclusion$)$.
\end{remark}

\subsection{Asymptotic results for $x\to \infty$} 
Next we give the results describing the relations between asymptotic behavior 
of $p_t(x)$, $q_t^*(x)$ and $f_t(x)$, when $x\to \infty$. In the first of them we show that there is a correspondence between 
asymptotic behavior of $p_t(x)/\nu(x)$ and $q_t^*(x)/\nu(x)$ as $x\to \infty$ and $t$ is bounded. 

\begin{theorem}
 \label{thm:asympt:x}
Assume that \eqref{eq:nu:monoton}, \eqref{eq:nu:double} hold and there exists $x_0>0$ such that for every  $a>x_0$ we have
	\begin{equation}
	\label{eq:nu:ratio:limit}
	\lim_{x\to \infty} \frac{\nu(x+a)}{\nu(x)}= 1\/.
	\end{equation}
	Let $t_0>0$. If 
		\begin{equation}
	   \label{eq:ptnu:asympt:xstrong}
		  \lim_{x\to \infty} \frac{p_t(x)}{t\nu(x)}=1\/,
	\end{equation}
	uniformly on $(0,t_0]$, then
			\begin{equation}
	   \label{eq:qtnu:asympt:xweak}
		  \lim_{x\to \infty} \frac{q_t^*(x)}{\nu(x)}={\tt d^*}+N^*(t)
	\end{equation}
	uniformly on $(0,t_0]$. Conversely, if  
			\begin{equation}
	   \label{eq:qtnu:asympt:xstrong}
		  \lim_{x\to \infty} \frac{q_t^*(x)}{({\tt d^*}+N^*(t))\nu(x)}=1
	\end{equation}
	uniformly on $(0,t_0]$, then
			\begin{equation}
	   \label{eq:ptnu:asympt:xweak}
		  \lim_{x\to \infty} \frac{p_t(x)}{\nu(x)}=t
	\end{equation}
	uniformly on $(0,t_0]$. Additionally, if we strengthen the condition \eqref{eq:nu:monoton} by assuming that $\nu(x)$ is 
	non-increasing on $(x^*,\infty)$ for some $x^*>0$, then \eqref{eq:ptnu:asympt:xstrong} and \eqref{eq:qtnu:asympt:xstrong} 
	are equivalent.
\end{theorem} 

\begin{remark}
   \label{rem:dstarpositive:xinfty}
   Note that \eqref{eq:ptnu:asympt:xstrong} and \eqref{eq:qtnu:asympt:xstrong} are equivalent in an obvious way if ${\tt d^*}>0$, 
   since then ${\tt d^*}+N^*(t)\sim {\tt d^*}$.
\end{remark}

The results of Theorem \ref{thm:asympt:x} can be applied to study the asymptotics of the supremum density at infinity, as it is 
done in the following theorem.

\medskip

\begin{theorem}
  \label{thm:ftx:xinfty}
	Assume that \eqref{eq:nu:monoton}, \eqref{eq:nu:double} and \eqref{eq:nu:ratio:limit} hold and let $t_0>0$. If 
		\begin{equation}
	   \label{eq:ptnu:asympt:strong2}
		  \lim_{x\to \infty} \frac{p_t(x)}{t\nu(x)}=1\/,
	\end{equation}
	uniformly on $(0,t_0]$, then
			\begin{equation}
	   \label{eq:ftnu:asympt:weak}
		  \lim_{x\to \infty} \frac{f_t(x)}{\nu(x)}=t
	\end{equation}
	uniformly on $(0,t_0]$. Moreover, if $\nu(x)$ is non-increasing on $(x^*,\infty)$ for some $x_0>0$, then
				\begin{equation}
	   \label{eq:ftnu:asympt:strong}
		  \lim_{x\to \infty} \frac{f_t(x)}{t\nu(x)}=1
	\end{equation}
	uniformly on $(0,t_0]$.
\end{theorem}

\section{Examples and applications}
\label{sec:examples}

In this section we present various examples in order to illustrate the wide range of possible applications of the above results. 

\subsection{Stable processes} The first example regards stable processes. It is easy to check that all the assumptions of the 
above-given results hold in this case. One can use the asymptotic results from \cite{bib:zol86} to verify the required uniform 
convergence of $p_t(x)/t$ and $p_t(x)/(t\nu(x))$ as $t$ goes to $0$ and $x\to \infty$ respectively. In particular we can recover 
results of \cite{bib:ds10} and \cite{ku1}, \cite{ku2}.

\subsection{Brownian motion with drift subordinated} In the next example we present a L\'evy process whose density of the L\'evy 
measure fulfills nice properties on the positive half-line (so that all assumptions of our results hold) but with less regularity on the 
negative half line. More precisely, we consider a Brownian motion with positive drift $Y_t=B_t+at$, $a>0$, subordinated by the 
independent $1/2$-stable subordinator $T^{(1/2)}$, i.e. the process $X_t = Y_{T_t}$. Then the transition density function $p_t(x)$ 
of $X_t$ is given by
\begin{eqnarray*}
   p_t(x) &=& \int_{0}^\infty \dfrac{1}{\sqrt{2\pi s}}e^{-\frac{(x-as)^2}{2s}}\eta_{t}^{(1/2)}(s)ds\/,\quad \eta_t^{(1/2)}(s) = 
   \frac{t}{\sqrt{2\pi}}\frac{1}{s^{3/2}}e^{-\frac{t^2}{2s}}\/.
\end{eqnarray*}
Using the integral representation 
\begin{equation*}
   K_\nu(z) = 2^{-\nu-1}z^\nu\int_0^\infty e^{-u}e^{-\frac{z^2}{4u}}u^{-\nu-1}du
\end{equation*}
of the modified Bessel function (see for example 8.432 (6) in \cite{bib:gr00}) we obtain
\begin{equation*}
   p_t(x) = \frac{te^{ax}}{2\pi}\int_0^\infty e^{-\frac{x^2+t^2}{2s}}e^{-\frac{a^2s}{2}}\frac{ds}{s^2} = \frac{ate^{ax}}{\pi}
   \frac{K_1(a\sqrt{x^2+t^2})}{\sqrt{x^2+t^2}}\/.
\end{equation*}
It is now straightforward that
\begin{equation*}
   \frac{p_t(x)}{t} \to \nu(x) = \frac{a\,e^{ax}}{\pi}\frac{K_1(a|x|)}{|x|}\/,
\end{equation*} 
uniformly on $\{|x|>\varepsilon\}$ for every given $\varepsilon>0$. One can easily show that using the asymptotic behavior 
$K_\nu(z)\cong \sqrt{\frac{\pi}{2z}}e^{-z}$, $z\to \infty$, together with $(zK_1(z))'=-zK_0(z)$, we obtain that $\nu(x)$ is 
non-increasing function on the positive half-line, it is regularly varying at $\infty$ with index $1+1/2$ and all the regularity conditions 
required in theorems given in the previous sections hold. However we stress the fact that $\nu(x)$ decays exponentially on the 
negative half-line. Thus, in particular, the condition \eqref{eq:nu:double} does not hold, but since we do not require any regularity 
on the negative half-line, we can apply our results here. 

Moreover, this example can be extended by considering general $\alpha/2$-stable subordinator $T^{(\alpha/2)}$ and the subordinated 
process $X_t = Y_{T_t^{(\alpha/2)}}$, with $Y_t= B_t+at$ as previously. Although the corresponding transition density function 
$\eta_t^{(\alpha/2)}(u)$ does not have such simple representation as in the $1/2$-stable case, using the scaling property 
$\eta_t^{(\alpha/2)}(u)=t^{-2/\alpha}\eta_1^{(\alpha/2)}(ut^{-2/\alpha})$ and the series representation (see formula 2.4.8 on page 90 in \cite{bib:zol86})
\begin{equation*}
    \eta_1^{(\alpha/2)}(u) = \frac{1}{\pi} \sum_{n=1}^\infty (-1)^{n-1}\frac{\Gamma(n\alpha/2+1)}{\Gamma(n+1)}\frac{\sin(\pi n\alpha/2)}
    {u^{1+n\alpha/2}}\/,\quad u>0\/,
\end{equation*}
one can show that 
\begin{equation*}
   \frac{p_t(x)}{t} \to \nu(x) = \dfrac{a^{\frac{1+\alpha}{2}}e^{ax}}{\pi}\frac{K_{\frac{1+\alpha}{2}}(a|x|)}{|x|^{\frac{1+\alpha}{2}}},
\end{equation*}
as $t\to 0$, uniformly on $\{|x|>\varepsilon\}$. Moreover, it can be also shown that $p_t(x)/(t\nu(x))$ tends to $1$, as $x\to \infty$, 
uniformly for small $t$. Regularity of the above-given L\'evy measure on positive half-line combined with our results give
\begin{corollary}
  Let $X_t = Y_{T_t^{(\alpha/2)}}$, where $T^{(\alpha/2)}$ is $\alpha/2$-stable subordinator, $\alpha\in(0,2)$, independent from a 
  Brownian motion with positive drift $Y_t=B_t+at$, $a>0$. Then for given $t_0,x_0>0$ we have
	\begin{equation}
   \lim_{t\to0^+}\frac{f_t(x)}{t}=\nu(x) = \dfrac{a^{\frac{1+\alpha}{2}}e^{ax}}{\pi}\frac{K_{\frac{1+\alpha}{2}}(a|x|)}{|x|^{\frac{1+\alpha}{2}}},
\end{equation}
uniformly on $[x_0,\infty)$ and
\begin{equation}
   \lim_{x\to \infty}\frac{f_t(x)}{t\nu(x)} = 1,
\end{equation}
uniformly on $(0,t_0]$. In particular
\begin{equation*}
   f_t(x)\stackrel{x_0,t_0}{\approx}  \dfrac{ta^{\frac{1+\alpha}{2}}e^{ax}}{\pi}\frac{K_{\frac{1+\alpha}{2}}(a|x|)}
   {|x|^{\frac{1+\alpha}{2}}}\/,\quad x\geq x_0\/,t\leq t_0\/.
\end{equation*}
\end{corollary} 

\subsection{Subordinator with negative drift} In order to give an example where $(0,\infty)$ is not regular (i.e. ${\tt d^*}>0$) and 
to which our results apply, let us consider the process $X_t = T^{(1/2)}_t-at$, $a>0$, where $T^{(1/2)}$ is an $1/2$-stable 
subordinator. Obviously this example can easily be extended to general $\alpha/2$-stable subordinator with negative drift, as it 
was done in the example given above. The details are left to the reader. Note that the corresponding transition density function 
is given by
\begin{equation*}
   p_t(x) = \dfrac{t}{\sqrt{2\pi}}\dfrac{1}{(x+at)^{3/2}}e^{-\frac{t^2}{2(x+at)}}\/,\quad x>0\/.
\end{equation*}
It is an easy exercise to check that $p_t(x)/t$ tends to $\nu(x)=x^{-3/2}/\sqrt{2\pi}$, as $t\to 0$, uniformly in $x>\varepsilon$, 
for every fixed $\varepsilon>0$ and $p_t(x)/(t\nu(x))$ tends to 1, as $x\to \infty$, uniformly for small $t$. Consequently all the 
results can be used to study asymptotics and estimates of the supremum density $f_t(x)$.

\subsection{Unimodal L\'evy process} We consider symmetric unimodal L\'evy processes with regularly varying 
L\'evy-Kchintchin exponent. Unimodality on a real line means that the density $p_t(\cdot)$ of $X_t$ under 
$\p$ exists on $\mathbb{R}\setminus\{0\}$ and that $p_t(\cdot)$ is symmetric and non-increasing on $(0,\infty)$. As it was 
shown in \cite{bib:w83}, this is equivalent to assume that the corresponding L\'evy measure has a symmetric density which is 
non-increasing on $(0,\infty)$. Using Theorem 7 and Corollary 2 from \cite{bib:cgt2017} combined together with Theorems  
\ref{thm:ftx:xinfty} and \ref{thm:ftx:tzero} above give the following.
\begin{theorem}
\label{thm:unimodal}
Assume that $(X,\p)$ is a symmetric unimodal L\'evy process whose L\'evy-Kchintchin exponent is regularly varying at zero, 
with index $\alpha\in(0,2)$. Then for any $t_0,x_0>0$,
\begin{equation}
   \lim_{x\to \infty}\frac{f_t(x)}{t\nu(x)} = 1,
\end{equation}
uniformly on $(0,t_0]$ and
\begin{equation}
   \lim_{t\to0^+}\frac{f_t(x)}{t}=\nu(x),
\end{equation}
uniformly on $[x_0,\infty)$. In particular
\begin{equation*}
   f_t(x)\stackrel{x_0,t_0}{\approx} t\nu(x)\/,\quad x\geq x_0\/,t\leq t_0\/.
\end{equation*}
\end{theorem}



\section{Proofs}\label{proofs}



This section is organized as follows. We begin with proving the auxiliary results, i.e. Theorem \ref{thm:qtpt:equalities},  Lemma 
\ref{lem:conv} introduced below and then Propositions \ref{prop:qt:tail} and \ref{prop:qt:general}. Then we give proofs of Theorems \ref{thm:asympt:t}-\ref{thm:unimodal}.

\subsection{Proofs of auxiliary results}\label{subsec:Aux}

\begin{proof}[Proof of Theorem $\ref{thm:qtpt:equalities}$] First recall the so-called Fristedt identity which relates the Laplace exponent 
$\kappa(\alpha,\beta)$ of $(\tau,H)$ 
to the law of $(X,\p_x)$, see Theorem 6.16 in \cite{ky}. For all $\alpha\ge0$ and $\beta\ge0$, this exponent 
is given by:
\begin{equation}\label{fris}
\kappa(\alpha,\beta)=\exp\left(\int_0^\infty dt\int_{0}^\infty(e^{-t}-e^{-\alpha t-\beta x})t^{-1}\,\p(X_t\in dx)\right)\,.
\end{equation}
(Note that the constant $k$ which appears in Theorem 6.16 of \cite{ky} is equal to $1$, according to our normalization (\ref{norm1})).
Then recall the definition (\ref{7342}) of $\kappa(\alpha,\beta)$. This expression is differentiable, in $\alpha>0$, in $\beta>0$ and 
in $u>0$. 

Differentiating both sides in $\beta$, we obtain:
\begin{equation}\label{6621}
\e(H_ue^{-\alpha\tau_u-\beta H_u})=u\,\e(e^{-\alpha\tau_u-\beta H_u})\frac{\partial}{\partial \beta}\kappa(\alpha,\beta)\,.
\end{equation}
Then since
\begin{eqnarray*}
\frac{\partial}{\partial \beta}\int_0^\infty dt\int_{0}^\infty(e^{-t}-e^{-\alpha t-\beta x})t^{-1}\,\p(X_t\in dx)&=&
\int_0^\infty dt\int_{0}^\infty e^{-\alpha t-\beta x}xt^{-1}\,\p(X_t\in dx)\\
&=&\int_0^\infty e^{-\alpha t}\e\left(X_te^{-\beta X_t}\ind_{\{X_t\ge0\}}\right)t^{-1}dt\,,
\end{eqnarray*}
we derive from (\ref{fris}) and (\ref{6621}) that,
\begin{eqnarray*}
\e(H_ue^{-\alpha\tau_u-\beta H_u})&=&-u\int_0^\infty e^{-\alpha t}\e\left(X_te^{-\beta X_t}\ind_{\{X_t\ge0\}}\right)t^{-1}dt
\frac{\partial}{\partial u}\e(e^{-\alpha\tau_u-\beta H_u})\\
&=&-u\frac{\partial}{\partial u}\int_0^\infty\e\left(\e(e^{-\alpha\tau_u-\beta H_u})t^{-1}e^{-\alpha t}X_t
e^{-\beta X_t}\ind_{\{X_t\ge0\}}\right)dt\,.
\end{eqnarray*}
Let $\tilde{X}$ be a copy of $X$ which is independent of $(\tau_u,H_u)$. Then the above expression may be written as:
\[\e(H_ue^{-\alpha\tau_u-\beta H_u})
=-u\frac{\partial}{\partial u}\e\left(\int_0^\infty \exp(-\alpha(t+\tau_u)-\beta(\tilde{X}_t+H_u))\ind_{\{\tilde{X}_t\ge0\}}
t^{-1}\tilde{X}_t\,dt\right)\,.\]
For $\tilde{X}$, we may take for instance
$\tilde{X}=(X_{\tau_u+t}-X_{\tau_u},\,t\ge0)$, so that it follows from a change of variables,
\begin{eqnarray*}
&&\e(H_ue^{-\alpha\tau_u-\beta H_u})\\
&=&-u\frac{\partial}{\partial u}\e\left(\int_0^\infty \exp(-\alpha(t+\tau_u)-\beta X_{\tau_u+t})
\ind_{\{X_{\tau_u+t}\ge H_u\}}t^{-1}(X_{\tau_u+t}-H_u)\,dt\right)\\
&=&-u\frac{\partial}{\partial u}\e\left(\int_0^\infty \exp(-\alpha t-\beta X_{t})\ind_{\{X_{t}\ge H_u,\,\tau_u\le t\}}
(t-\tau_u)^{-1}(X_{t}-H_u)\,dt\right)\,,
\end{eqnarray*}
from which we deduce through an integration by part that 
\begin{eqnarray*}
&&\int_0^\infty \e(H_ue^{-\alpha\tau_u-\beta H_u})\,du\\
&=&\int_0^\infty \e\left(\int_0^\infty \exp(-\alpha t-\beta X_{t})\ind_{\{X_{t}\ge H_u,\,\tau_u\le t\}}
(t-\tau_u)^{-1}(X_{t}-H_u)\,dt\right)\,du\,.
\end{eqnarray*}
Applying the strong Markov property at time $\tau_u$, we obtain
\begin{eqnarray*}
&&\int_0^\infty \e\left(\int_0^\infty \exp(-\alpha t-\beta X_{t})\ind_{\{X_{t}\ge H_u,\,\tau_u\le t\}}
(t-\tau_u)^{-1}(X_{t}-H_u)\,dt\right)\,du\\
&=&\int_0^\infty \e\left(\int_0^\infty \e_{H_u}(e^{-\alpha t-\beta X_{t-s}}\ind_{\{X_{t-s}\ge z\}}(X_{t-s}-z))_{z=H_u,s=\tau_u}\ind_{\{\tau_u\le t\}}
(t-\tau_u)^{-1}\,dt\right)\,du\,,\\
&=&\int_{t=0}^\infty\int_{x=0}^\infty e^{-\alpha t-\beta x}\int_{s=0}^t\int_{z=0}^x\frac{x-z}{t-s}\p_z(X_{t-s}\in dx)g(ds,dz)\,dt
\end{eqnarray*}
and (\ref{2294}) follows by inverting the Laplace transforms.

Then (\ref{2681}) is obtained in the same way by differentiating  (\ref{fris}) with respect to $\alpha$.
\end{proof}


\noindent Next we state and prove the following result, which will be used later in the proof of Theorem \ref{thm:qtx:bounds}.

\begin{lemma}
\label{lem:conv}
 Let $X$ be any real L\'evy process which is not a compound Poisson process and such that $|X|$ is not a subordinator. Then
\begin{equation}
   \label{eq:conv1}
   \lim_{x\to 0^+} \frac{\p_x(\tau_0^->t)}{h^*(x)} = n^*(t<\zeta)\/,\quad 
\end{equation}
and
\begin{equation}
  \label{eq:conv2}
 \limsup_{x\to 0^+}\frac{1}{h^*(x)}\int_0^t \p_x(\tau_0^->s)ds \leq \left(1+e\right) \left({\tt d^*}+N^*(t)\right)\/.
\end{equation}
\end{lemma}
\begin{proof}
The relation \eqref{eq:conv1} was proved in the proof of Proposition 3 from \cite{bib:chm16}. To deal with \eqref{eq:conv2} we recall 
formula (4.4) from \cite{bib:chm16} (see also formula (1) from \cite{bib:chd08}) rewritten in the following form
\begin{equation*}
   \int_0^\infty e^{-\varepsilon s}\p_x(\tau_0^->s)ds = \e\left(\int_0^\infty e^{-\varepsilon s}\ind_{\{\underline{X}_s\geq -x\}}
   dL^*_s\right)[{\tt d^*}+\int_0^\infty e^{-\varepsilon s} n^*(s<\zeta)ds]
\end{equation*}
with $\varepsilon>0$. Since 
\begin{equation*}
   \e\left(\int_0^\infty e^{-\varepsilon s}\ind_{\{\underline{X}_s\geq -x\}}dL^*_s\right) \leq 
   \e\left(\int_0^\infty \ind_{\{\underline{X}_s\geq -x\}}dL^*_s\right)=h^*(x)
\end{equation*}
and 
\begin{eqnarray*}
   \int_0^\infty e^{-s/t} n^*(s<\zeta)ds \leq \int_0^t n^*(s<\zeta)ds+n^*(t<\zeta)\int_t^\infty e^{-s/t}ds \leq (1+e^{-1})N^*(t)
\end{eqnarray*}
we get, by setting $\varepsilon=1/t$, that
\begin{eqnarray*}
   \int_0^t \p_x(\tau_0^->s)ds &\leq& e\int_0^\infty e^{-s/t}\p_x(\tau_0^->s)ds\\
	&\leq& e\,h^*(x)[{\tt d^*}+\int_0^\infty e^{-\varepsilon s} n^*(s<\zeta)ds]\leq(1+e)h^*(x)[{\tt d^*}+N^*(t)]\/.
\end{eqnarray*}
Dividing both sides by $h^*(x)$ and taking the limit end the proof.
\end{proof}


\begin{proof}[Proof of Proposition $\ref{prop:qt:tail}$]
It is enough to prove the bounds for $q_t^*((x,\infty))$. We use monotonicity of $h^*(x)$ and get
\begin{equation*}
   q_t^*((x,\infty)) = n^*(X_t>x,t<\zeta) = \p^{\uparrow}(X_t>x,1/h^*(X_t))\leq \dfrac{n^*(t<\zeta)}{h^*(x)}\dfrac{\p^{\uparrow}(X_t>x)}
   {n^*(t<\zeta)}\/,
\end{equation*}
where $\p^{\uparrow}$ stands for the law of the L\'evy process $(X,\p)$ conditioned to stay positive. To prove the first estimate note 
that 
\begin{equation*}
   \p^{\uparrow}(X_t>x)\leq \p^{\uparrow}(X_t>x_0) \to 0\/,\quad t\to 0^+\/,
\end{equation*}
by the right continuity of the paths. Since $t\to n^*(t<\zeta)$ is non-increasing we can make the expression 
$\p^{\uparrow}(X_t>x)/n^*(t<\zeta)$ uniformly small in $t<t_0$ and $x>x_0$ by choosing $t_0$ small enough. If ${\tt d}=0$, 
then $n^*(t<\zeta)$ tends to infinity as $t\to0^+$ and consequently for every $x>0$ we have
\begin{equation*}
   \frac{\p^{\uparrow}(X_t>x)}{n^*(t<\zeta)}\leq \dfrac{1}{n^*(t<\zeta)} \to 0\/,\quad t\to 0^+\/.
\end{equation*}
To prove the other bounds we write once again
\begin{equation*}
   q_t^*((x,\infty)) = n^*(X_t>x,t<\zeta) \leq \dfrac{n^*(t<\zeta)}{h^*(x)}\dfrac{\p^{\uparrow}(X_t>x)}{n^*(t_0<\zeta)}\/,
\end{equation*}
for given $t_0>0$. Using the right continuity of the paths and monotonicity of the function $x\to \p^{\uparrow}(X_t>x)$ we can show that 
\begin{equation*}
   \lim_{t\to t_1^+,x\to \infty}\p^{\uparrow}(X_t>x) = 0,
\end{equation*}
for every $t_1\in[0,t_0]$. Since $[0,t_0]$ is compact we deduce that 
\begin{equation*}
  \lim_{x\to \infty}\p^{\uparrow}(X_t>x) = 0,
\end{equation*}
uniformly on $(0,t_0)$. This ends the proof. 
\end{proof}


\begin{proof}[Proof of Proposition $\ref{prop:qt:general}$]
Let $x_0,t_0>0$ as in the assumptions of the proposition. We take $r=x_0/4$ and denote by $\tau_r=\inf\{t\geq 0: X_t>r\}$ the first 
passage time above level $r$. For every $x<r$, using the strong Markov property, we write
		\begin{eqnarray}
		   \qquad q_t^*(x,y) &=& \int_{u=0}^{u=t}\int_{z=r}^\infty\mathbb{Q}_x^*(\tau_r\in du, X_{\tau_r}\in dz)
		   q_{t-u}^*(z,y)\leq I_t^{(1)}(x,y)+I^{(2)}_t(x,y)\/, 
								\label{eq:qtxy:upper}
		\end{eqnarray}
		where
		\begin{eqnarray}
		   \label{eq:qtxy:I1}
		   I_t^{(1)}(x,y) &=& \int_{u=0}^{u=t} \int_{(r,y/2)} \mathbb{Q}_x^*(\tau_r\in du,X_{\tau_r}\in dz)\,p_{t-u}(y-z)\/,\\
			\label{eq:qtxy:I2}
			 I_t^{(2)}(x,y) &=& \int_{u=0}^{u=t} \int_{(y/2,\infty)} \mathbb{Q}_x^*(\tau_r\in du,X_{\tau_r}\in dz)\,q_{t-u}^*(z,y)\/.
		\end{eqnarray}
	   Using boundedness of $p_t(x)/t$ and $\nu(x)$ we can find positive constant $c>0$ such that for $y>3x_0=:x_1$ 
	   and $t<t_0$ we have
		\begin{eqnarray*}
		 I^{(1)}_t(x,y) &\leq& c\int_{u=0}^{u=t}(t-u) \int_{(r,y/2)} \mathbb{Q}_x^*(\tau_r\in du,X_{\tau_r}\in dz)\\
		&\leq& ct \mathbb{Q}_x^*(\tau_r<t,X_{\tau_r}\in (r,y/2))\leq ct \p_x(\tau_0^->t)\/,
		\end{eqnarray*}
		where we have used the fact that $y-z\geq y-y/2\geq x_0$ for $z\in(r,y/2)$. Moreover, since $w<r<2r<z$ and $u<t$, we have
\begin{eqnarray*}
    \mathbb{Q}_x^*(\tau_r\in du,X_{\tau_r-}\in dw,X_{\tau_r}\in dz) = q_u^*(x,w)\nu(z-w)\,du\,dw\,dz\/,
\end{eqnarray*}		
which follows from the Ikeda-Watanabe formula. Note that the condition $w<r<2r<z$ implies that $X_{\tau_r-}\neq X_{\tau_r}$. 
Thus we have		
\begin{eqnarray*}
   I^{(2)}_t(x,y) &=& \int_{u=0}^{u=t} \int_{w=0}^{w=r} \int_{(y/2,\infty)} \mathbb{Q}_x^*(\tau_r\in du,X_{\tau_r-}\in dw, 
   X_{\tau_r}\in dz)\,q_{t-u}^*(z,y)\\
	&=& \int_{0}^{t}du \int_{0}^{r}dw \int_{(y/2,\infty)} q_u^*(x,w)\nu(z-w)q_{t-u}^*(z,y)\,dz.
\end{eqnarray*}
For $z\in(y/2,\infty)$ and $w\in(0,r)$ we get $z-w\geq y/2-r\geq 3/2x_0-x_0/4\geq x_0$ and consequently
		\begin{eqnarray*}
			I^{(2)}_t(x,y) &\leq& c\,\int_0^t \mathbb{Q}_x^*(X_u\in(0,r))
			\mathbb{Q}_y(X_{t-u}\in(y/2,\infty))du\leq c\,\int_0^t \p_x(\tau_0^+>u)du\/.
		\end{eqnarray*}
			Using \eqref{eq:conv1} and \eqref{eq:conv2} from Lemma \ref{lem:conv} and the fact that $tn^*(t<\zeta)\leq N^*(t)$, 
			which follows from monotonicity of $t\to n^*(t<\zeta)$, we obtain
				$$
	  q_t^*(y) = \lim_{x\to 0^+}\frac{q_t^*(x,y)}{h^*(x)}\leq c ({\tt d^*}+N^*(t))\/.
	$$
	Conversely, assuming boundedness of $q_t^*(x)/({\tt d^*}+N^*(t))$ the boundedness of $p_t(x)/t$ follows directly from \eqref{eq:pt:formula:exp}, since for $x$ large enough we have
	\begin{eqnarray*}
	   p_t(x) &=& \int_0^tdu \int_0^\infty q_u^*(x+z)q_{t-u}(dz)+{\tt d^*}q_t^*(x)\\
		&\leq & c\left[\int_0^t n(t-u<\zeta)N^*(u)\,du+{\tt d}N^*(t)+{\tt d^*}N(t)\right] = ct\/,
	\end{eqnarray*}
	where the equality
   	\begin{eqnarray}
	   \label{eq:n:relation}
		   \int_0^t n(t-u<\zeta)N^*(u)\,du +{\tt d}N^*(t)+{\tt d^*}N(t)= t\/,
		\end{eqnarray}	
	follows from \eqref{eq:sup:formula} in the following way 
		\begin{eqnarray*}
	   t &=& \int_0^tdu\int_{[0,\infty)}\p(\overline{X}_u\in dx)\\
			&=& \int_0^tdu \int_0^un(s<\zeta)n^*(u-s<\zeta)\,ds+{\tt d}N^*(t) +{\tt d^*}N(t)\\
			&=& \int_0^t n(s<\zeta)\left(\int_s^t n^*(u-s<\zeta)\,du\right)\,ds+{\tt d}N^*(t) +{\tt d^*}N(t)\\
			&=& \int_0^t n(s<\zeta)N^*(t-s<\zeta)\,ds+{\tt d}N^*(t) +{\tt d^*}N(t)\/.
  \end{eqnarray*}	
	Note that we have used the fact that ${\tt dd^*}=0$. This ends the proof.  
	\end{proof}

\subsection{Proofs of asymptotic results as $t\to0^+$}
\begin{proof}[Proof of Theorem $\ref{thm:asympt:t}$]
We begin with assuming \eqref{eq:ptnu:asympt:tweak} with some $x_0>0$ and fixing $\varepsilon>0$. First, using the uniform continuity 
of $\nu(x)$ we find $\delta>0$ such that
$$
 |\nu(x)-\nu(y)|\leq \varepsilon/5
$$
whenever $|x-y|<\delta$ and $x,y>x_1$. We can also require that $\delta<\varepsilon/(5x_0)$. Then, for every $x>x_2=(x_0\vee x_1)+\delta$ 
we use \eqref{eq:ACformula} and write
\begin{eqnarray*}
   \frac{q_t^*(x)}{{\tt d^*}+N^*(t)}\lefteqn{-\nu(x) = \frac{1}{{\tt d^*}+N^*(t)}\int_0^t du \int_{0}^{\delta}q_u^*(z)\left(1-\frac{z}{x}\right)
   \left(\frac{p_{t-u}(x-z)}{t-u}-\nu(x-z)\right)dz}\\
	&&+\frac{{\tt d^*}}{{\tt d^*}+N^*(t)}\left(\dfrac{p_t(x)}{t}-\nu(x)\right)\\
	&&+\frac{1}{{\tt d^*}+N^*(t)}\int_0^t du \int_0^\delta q_u^*(z)\left(1-\frac{z}{x}\right)(\nu(x-z)-\nu(x))dz\\
	&&+\frac{1}{{\tt d^*}+N^*(t)}\int_0^t du \int_{\delta}^{x}q_u^*(z)\left(1-\frac{z}{x}\right)\frac{p_{t-u}(x-z)}{t-u}dz\\
	&&-\frac{\nu(x)}{x({\tt d^*}+N^*(t))}\int_0^t du \int_0^\delta zq_u^*(z)dz-\frac{\nu(x)}{{\tt d^*}+N^*(t)}\int_0^t du \int_{\delta}^\infty 
	q_u^*(z)dz\/.
\end{eqnarray*}
Taking $t$ small enough, by \eqref{eq:ptnu:asympt:tweak}, we can estimate the sum of the absolute value of the first integral above 
and the succeeding expression by $\varepsilon/5$ uniformly for $x>x_2$. We also simply have
\begin{eqnarray*}
\left|\frac{1}{{\tt d^*}+N^*(t)}\int_0^t du \int_0^\delta q_u^*(z)\left(1-\frac{z}{x}\right)(\nu(x-z)-\nu(x))dz\right| \leq \frac{\varepsilon}{5}\/,
\end{eqnarray*}
for every $x>x_2$. Using \eqref{eq:qt:tail} we get
$$
  \frac{\nu(x)}{{\tt d^*}+N^*(t)}\int_0^t du \int_{\delta}^\infty q_u^*(z)dz \leq \frac{\varepsilon}{5}\nu(x)
$$
for $t$ small enough and the $\delta<\varepsilon/(5x_0)$ gives
$$
  \frac{\nu(x)}{x({\tt d^*}+N^*(t))}\int_0^t du \int_0^\delta zq_u^*(z)dz \leq \frac{\nu(x)\delta}{x_0} \leq \frac{\varepsilon}{5}\nu(x)\/.
$$
Finally, the remaining part 
$$ 
I_3(t,x) = \frac{1}{{\tt d^*}+N^*(t)}\int_0^t du \int_{\delta}^{x}q_u^*(z)\left(1-\frac{z}{x}\right)\frac{p_{t-u}(x-z)}{t-u}dz
$$
can be bounded as follows. Uniform continuity of $\nu(x)$, which is the density of a finite measure, implies that $\nu(x)$ is bounded 
for large $x$. Since $p_t(x)/t$ converges uniformly to $\nu(x)$, $p_t(x)/t$ is also bounded (by $c_1>0$) for $x>x_3$ and small $t$. 
Thus we can apply Proposition \ref{prop:qt:general} to get
$$
  q_t^*(x)\leq c_2\,({\tt d^*}+N^*(t))
$$
for $x>x_4$ and $t$ small. Thus for $x>2x_5$ with $x_5 = (x_3\vee x_4)$ we have
\begin{eqnarray*}
 I_3(t,x) &\leq&\frac{1}{{\tt d^*}+N^*(t)}\int_0^t du \left[\int_{\delta}^{x_3}+\int_{x_3}^{x}\right]q_u^*(z)\left(1-\frac{z}{x}\right)
 \frac{p_{t-u}(x-z)}{t-u}dz\\
 &\leq&  \frac{c_1}{{\tt d^*}+N^*(t)}\int_0^t du \int_{\delta}^{x_3}q_u^*(z)dz+ c_2\int_0^t du\int_{x_3}^{x}\left(1-\frac{z}{x}\right)
 \frac{p_{t-u}(x-z)}{t-u}dz\\
&\leq& c_1 t + c_2 \int_0^t du \int_{0}^{x-x_3}\frac{w}{x} \frac{p_u(w)}{u}dw\/.
\end{eqnarray*}
Here we used also Proposition \ref{prop:qt:tail} to show that $\int_{\delta}^{x_3}q_u^*(z)dz\leq 1$ for $u<t$ with $t$ sufficiently small. 
Moreover, we have
\begin{eqnarray*}
 \int_0^t du\int_{0}^{x-x_3}\frac{w}{x} \frac{p_u(w)}{u}dw &\leq& \frac{1}{x}\int_0^t du\int_{0}^{1}(1\wedge(|w|+u))\frac{p_u(w)}{u}dw+
 \int_0^t du\int_{0}^{1}\frac{p_u(w)}{u}dw\\
&\leq& \left(1\vee \frac{1}{x_5}\right)e^{qt}\int_0^t du\int_0^\infty (1\wedge (|w|+u))\frac{p_u(w)}{u}e^{-qu}dw\/.
\end{eqnarray*}
To finish the estimates recall that
$$ 
  \frac{p_t(dz)}{t}e^{-qt}
$$
is the L\'evy measure of the infinity divisible random variable $(\tau,X_\tau)$, where $\tau$ is an exponentially distributed random 
variable with parameter $q$, which is independent from $X$ (see \cite{bib:b96}, Lemma 7, Ch. VI, p.162). In particular we have
\begin{equation}
   \label{eq:levymeasure01}
  \int_0^\infty \int_{\mathbb{R}}(1\wedge (t+|z|))\frac{p_t(dz)}{t}e^{-qt}dt<\infty\/.
\end{equation}
Thus we can find $t_1>0$ such that for $t<t_1$ we have
$$
  I_3(t,x)\leq \frac{\varepsilon}{5}\/,\quad x>x_5\/.
$$
Collecting all together, we obtain
$$
  \left|\frac{q_t^*(x)}{{\tt d^*}+N^*(t)}-\nu(x)\right|\leq \varepsilon(1\vee \nu(x))
$$
for every $x>x_5$ and $t<t_1$. Since $\nu(x)$ is bounded the result follows. 

Assuming that \eqref{eq:qtnu:asympt:tweak} holds with $x_0>0$ we use \eqref{eq:pt:formula:exp} and \eqref{eq:n:relation} to write
\begin{eqnarray*}
  \frac{p_t(x)}{t}-\nu(x) &=& \frac{1}{t}\int_0^t du \int_0^\infty \left(\frac{q_u^*(x+z)}{{\tt d^*}+N^*(u)}-\nu(x+z)\right)({\tt d^*}+N^*(u))q_{t-u}(dz)\\
&& + \frac{1}{t}\int_0^t du \left(\int_0^\delta+\int_\delta^\infty\right) \left(\nu(x+z)-\nu(x)\right)({\tt d^*}+N^*(u))q_{t-u}(dz)\\
&&+\frac{{\tt d}}{t}\left(\frac{q_t^*(x)}{{\tt d^*}+N^*(t)}-\nu(x)\right)({\tt d^*}+N^*(t))\/,
\end{eqnarray*}
where $\delta$ is chosen as previously. Let $x>x_2=x_0\wedge x_1$. Using \eqref{eq:qtnu:asympt:tweak} we get that the first integral 
above together with the last term are bounded by
\begin{eqnarray*}
   \frac{\varepsilon}{2t}\left(\int_0^t du \int_0^\infty N^*(u)q_{t-u}(z)dz+{\tt d}N^*(t)+{\tt d^*}N(t)\right) = \frac{\varepsilon}{2}\/,
\end{eqnarray*}
for every $x>x_2$, where we used once again \eqref{eq:n:relation} and the fact that ${\tt dd^*}=0$. Then the integral
\begin{equation*}
   \frac{1}{t}\int_0^t du \int_0^{\delta} \left|\nu(x+z)-\nu(x)\right|({\tt d^*}+N^*(u))q_{t-u}(z)dz
\end{equation*}	
is uniformly bounded on $x>x_2$ by 
\begin{equation*}
\frac{\varepsilon}{5t}\left(\int_0^t du \int_0^\infty q_{t-u}(z)N^*(u)dz+{\tt d^*}N(t)\right)\leq \frac{\varepsilon}{5}\/,
\end{equation*}
where in the last inequality we used \eqref{eq:n:relation}. Finally, the boundedness of $\nu(x)$ on $(x_2,\infty)$ together with 
Proposition \ref{prop:qt:tail} entails that the remaining integral
\begin{equation*}
   \frac{1}{t}\int_0^t du \int_{\delta}^\infty \left|\nu(x+z)-\nu(x)\right|({\tt d^*}+N^*(u))q_{t-u}(z)dz
\end{equation*}
is bounded by 	
\begin{equation*}
	\frac{\varepsilon}{5t}\left(\int_0^t N^*(u)n(t-u<\zeta)du+{\tt d^*}N(t)\right)\leq \frac{\varepsilon}{5}\/,
\end{equation*}
by making $t$ sufficiently small. Collecting all together we obtain the result.
\end{proof}


\begin{proof}[Proof of Theorem $\ref{thm:ftx:tzero}$]
By \eqref{eq:ft:formula} we can write
	\begin{eqnarray*}
	   \frac{f_t(x)}{t}-\nu(x) &=& \frac{1}{t}\int_0^t\left(\frac{q_u^*(x)}{{\tt d^*}+N^*(u)}-\nu(x)\right)({\tt d^*}+N^*(u))n(t-u<\zeta)du\\
		&&+{\tt d}\left(\frac{q_t^*(x)}{{\tt d^*}+N^*(t)}-\nu(x)\right)({\tt d^*}+N^*(t))\/.
	\end{eqnarray*}
	By Theorem \ref{thm:asympt:t}, we have
	\begin{equation*}
	  \frac{q_u^*(x)}{{\tt d^*}+N^*(u)}\to\nu(x)\/,\quad t\to 0^+\/,
	\end{equation*}
	uniformly on $x>x_0$ and the required convergence follows from \eqref{eq:n:relation}. 
\end{proof}

\subsection{Proofs of estimates}

Since Proposition \ref{prop:qt:tail} was proved in Subsection \ref{subsec:Aux}, we begin this part directly with the proof of Theorem \ref{thm:qtx:bounds}.

  \begin{proof}[Proof of Theorem $\ref{thm:qtx:bounds}$]
	   We begin with the implication (i) $\Rightarrow$ (ii) and show the upper bounds in \eqref{eq:qtx:bounds}. For fixed $x_0,t_0>0$, 
	   we set $r=x_0/4$ and we use the same arguments as in the proof of Proposition \ref{prop:qt:general} in order to write
		$$
		  q_t^*(x,y)\leq I^{(1)}_t(x,y)+I^{(2)}_t(x,y)\/,
		$$
		where $I^{(1)}_t(x,y)$ and $I^{(2)}_t(x,y)$ are given by \eqref{eq:qtxy:I1} and \eqref{eq:qtxy:I2}, respectively. To deal with 
		$I^{(1)}_t(x,y)$ we note that since $z<y/2$, we have $|z-y|>y/2>2r$. Consequently, \eqref{eq:nu:monoton} and 
		\eqref{eq:nu:double} give $\nu(y-z)\leq c(r)\nu(y/2)\leq c^2(r)\nu(y)$. Thus using the upper bounds in \eqref{eq:ptnu:bounds} 
		we get 
\begin{eqnarray*}
    p_{t-u}(y-z)\leq c_1(t-u)\nu(y-z)\leq c_2\,t\nu(y)\/,\quad t<t_0\/, 
\end{eqnarray*}
for some $c_2=c_2(t_0,x_0)>0$. Consequently, we can write
\begin{eqnarray*}
  I^{(1)}_t(x,y) &\leq& c_2\,t\,\nu(y) \mathbb{Q}_x^*(X_{\tau_r}\in(r,y/2),\tau_r<t) \leq c_2 t\,\nu(y)\p_x(\tau_0^{-}>t)\/.
\end{eqnarray*}

To estimate $I^{(2)}_t(x,y)$ we will use the fact that for $w<r<2r<z$ and $u<t$ we have
\begin{eqnarray*}
    \mathbb{Q}_x^*(\tau_r\in du,X_{\tau_r-}\in dw,X_{\tau_r}\in dz) = q_u^*(x,w)\nu(z-w)\,du\,dw\,dz\/.
\end{eqnarray*}		
which follows from the Ikeda-Watanabe formula. Note that the condition $w<r<2r<z$ implies that $X_{\tau_r-}\neq X_{\tau_r}$. 
Thus we have		
\begin{eqnarray*}
   I^{(2)}_t(x,y) &=& \int_{u=0}^{u=t} \int_{w=0}^{w=r} \int_{(y/2,\infty)} 
   \mathbb{Q}_x^*(\tau_r\in du,X_{\tau_r-}\in dw, X_{\tau_r}\in dz)\,q_{t-u}^*(z,y)\\
	&=& \int_{0}^{t}du \int_{0}^{r}dw \int_{(y/2,\infty)} q_u^*(x,w)\nu(z-w)q_{t-u}^*(z,y)\,dz\/.
\end{eqnarray*}
Once again, since $z-w>y/2-r>y/4>r$ assumptions \eqref{eq:nu:monoton} and \eqref{eq:nu:double} imply that
$$
  \nu(z-w)\leq c_3 \nu(y/4)\leq c_4 \nu(y)\/.
$$
It gives
\begin{eqnarray*}
   I^{(2)}_t(x,y) &\leq& c_4\,\nu(y) \int_0^t \p_x(\tau_0^->u)du \int_{(y/2,\infty)}q_{t-u}^*(z,y)dz\\
	 &\leq& c_4\,\nu(y)\int_0^t \p_x(\tau_0^->u)du\/,
\end{eqnarray*}
where the last inequality simply follows from estimating $q_{t-u}^*(z,y)$ by $p_{t-u}(z,y)$ and enlarging the integration to the 
whole real line. Collecting all together we obtain
\begin{equation}
  \frac{q_t^*(x,y)}{h^*(x)}\leq c_5 \frac{\nu(y)}{h^*(x)}\left(\p_x(\tau_0^->t)+\int_0^t \p_x(\tau_0^->u)du\right)
\end{equation}
and Lemma \ref{lem:conv} gives
\begin{eqnarray*}
    q_t^*(y) = \lim_{x\to 0^+}\frac{q_t^*(x,y)}{h^*(x)}\leq c_6\, \nu(y)({\tt d^*}+N^*(t))\/,\quad y>x_0\/, t<t_0\/.
\end{eqnarray*}
The lower bounds can be shown as follows. For every $x>x_0$ and $t\leq t_0$, using \eqref{eq:ACformula}, we write
	\begin{eqnarray*}
	   \dfrac{q_{t}^*(x)}{\nu(x)}\geq {\tt d^*}\dfrac{p_t(x)}{t\nu(x)}+\int_{0}^{t} du \int_0^{x/2} q_u^*(z)\left(1-\frac{z}{x}\right)
	   \dfrac{p_{t-u}(x-z)}{(t-u)\nu(x-z)}\dfrac{\nu(x-z)}{\nu(x)}\,dz\/.
	\end{eqnarray*}
	We can estimate the last two quotients in the above-written integral using the lower bounds in \eqref{eq:ptnu:bounds} and \eqref{eq:nu:monoton}
	\begin{eqnarray*}
	  \dfrac{q_{t}^*(x)}{\nu(x)} &\geq& c_9\left({\tt d^*}+ \int_0^t du \int_0^{x/2} q_{u}^*(z)dz\right)\\
		&=& c_9 \left({\tt d^*}+\left(N^*(t)-\int_0^t \int_{x/2}^\infty q_u^*(z)dz\,du\right)\right)\/.
	\end{eqnarray*}
	Using the general upper bounds \eqref{eq:qt:tail} with $\varepsilon=h^*(x_0/2)/2$ we obtain
	\begin{eqnarray*}
	   \int_0^t \int_x^\infty q_u^*(z)dz du\leq \frac{h^*(x_0/2)}{2h^*(x/2)} N^*(t) \leq \frac{1}{2} N^*(t) \/,
		\end{eqnarray*}
		for $x>x_0$ and $t<t_1$, where the last inequality follows from monotonicity of $h^{*}(x)$. The same argument but 
		using \eqref{eq:qt:tail2} instead of \eqref{eq:qt:tail} leads to the desired lower bounds for $x>x_1$ and $t<t_0$, where 
		$x_1$ is some constant greater than $x_0$. To finish the proof of the implication (i) $\Rightarrow$ (ii) observe that  the 
		function
	$$
	  g(t,x) = \dfrac{\int_0^t du\int_0^x q_u^*(z)dz}{{\tt d^*}+N^*(t)}
	$$
is a strictly positive continuous function of $t>0$ and $x>0$. Thus it is bounded from below by a strictly positive constant on 
the compact set $[t_1,t_0]\times[x_0,x_1]$.
		
		\medskip
		
	Now assume that (ii) holds. Since for every $z>0$ and $x>x_0$ we have
	\begin{eqnarray*}
	   q_s^*(x+z)\leq c\nu(x+z)\left({\tt d^*}+N^*(t)\right)\leq c_2\nu(x)\left({\tt d^*}+N^*(t)\right)\/,
	\end{eqnarray*}
	where the last inequality follows from \eqref{eq:nu:monoton}, we get by \eqref{eq:pt:formula:exp} and \eqref{eq:n:relation} 
	that
	\begin{eqnarray*}
	   p_t(x) \leq c_2 \nu(x) \left(\int_0^t n(t-s<\zeta)N^*(t)ds +{\tt d}N^*(t)+{\tt d^*}N(t)\right) = c_2t\nu(x).
	\end{eqnarray*}
	for $x>x_0$ and $t<t_1$ with some $t_1\leq t_0$. Finally, to deal with the lower bounds we use once again 
	\eqref{eq:pt:formula:exp} and write for $x>x_0$ that
	\begin{eqnarray*}
	  p_t(x) &\geq & \int_0^t \int_0^x q_u^*(x+z)q_{t-u}(dz)du+{\tt d}q_t^*(x)\\
		&\geq& c_{10}\left[ \int_0^t \left({\tt d^*}+N^*(u)\right)\,\int_0^\infty \nu(x+z) q_{t-u}(dz)du+ {\tt d}\nu(x)({\tt d^*}+N^*(t))\right]\\
		&\geq & c_{10}\nu(x)\left[ \int_0^t N^*(u)\,\int_0^x q_{t-u}(dz)du+ {\tt d}N^*(t)+{\tt d^*}N(t)\right]\/, 
	\end{eqnarray*}
	where the second line is a consequence of (ii) and in the last estimate we used \eqref{eq:nu:double}. Using the general 
	upper bounds \eqref{eq:qt:tail} for $q_{t-u}((x,\infty))$ allows us to write
	\begin{eqnarray*}
	  \int_0^t \int_0^u n^*(\zeta<s)ds\,\int_0^x q_{t-u}(dz) \geq  \frac{1}{2}\int_0^t n(t-u<\zeta)N^*(u)du
	\end{eqnarray*}
	for $t<t_1$ and $x>x_0$, which by \eqref{eq:n:relation} gives
	$$
	   p_t(x)\geq ct\nu(x)
	$$
	for $t<t_1$ and $x>x_0$. The same holds true for $t<t_0$ and $x>x_1$ by using \eqref{eq:qt:tail2} instead of 
	\eqref{eq:qt:tail}. The estimates on the remaining compact set $[t_1,t_0]\times [x_0,x_1]$ follows from the continuity 
	argument as above. This ends the proof. 
 
	\end{proof}

		
	\begin{proof}[Proof of Theorem $\ref{thm:supremum:bounds}$]
	  The result is just an easy consequence of \eqref{eq:ft:formula}, the consecutive parts of the proof of Theorem \ref{thm:qtx:bounds}  
	  (see Remark \ref{rem:thmconditions} for detailed description) and the relation \eqref{eq:n:relation}.
\end{proof}

\subsection{Proofs of asymptotic results as $x\to \infty$}

\begin{proof}[Proof of Theorem $\ref{thm:asympt:x}$]

We begin with assuming \eqref{eq:ptnu:asympt:xstrong} and we will show that \eqref{eq:qtnu:asympt:xweak} holds. We fix $\varepsilon>0$ 
and $x_0>0$. Using \eqref{eq:nu:monoton} and \eqref{eq:nu:double} we find positive constants $c_1, c_2>0$ depending on $x_0$ such 
that 
\begin{equation}
   \label{eq:asympt:double}
	  \nu(z)\leq c_1 \nu(x/2)\leq c_2\nu(x)
\end{equation} 
for every $x>x_0$ and $z\in(x/2,x)$. Note also that \eqref{eq:ptnu:asympt:xstrong} ensures that \eqref{eq:ptnu:bounds} is fulfilled and consequently we can apply the result of Theorem \ref{thm:qtx:bounds} to get existence of a constant $c_3=c_3(t_0,x_0)>0$ such that
\begin{equation}
   \label{eq:asympt:qtu:upper}
	q_t^*(x)\leq c_3N^*(t)\nu(x)
\end{equation}
for every $t<t_0$ and $x>x_0$.
Writing
\begin{equation*}
  \int_0^{t_0}du \int_{0}^{(1-\delta)x}\frac{w}{x}\frac{p_{u}(w)}{u}dw \leq \frac{1}{x}\int_0^{t_0}du\int_0^1 (1\wedge w) 
  \frac{p_u(w)}{u}dw+(1-\delta)\int_0^{t_0}du\int_{1}^\infty \frac{p_u(w)}{u}dw
\end{equation*}
and using \eqref{eq:levymeasure01} we choose $\delta\in(1/2,1)$ and $x_1>x_0$ such that 
\begin{equation}
\label{eq:ptint:bounds}
\int_0^{t_0}du \int_{0}^{(1-\delta)x}\frac{w}{x}\frac{p_{u}(w)}{u}dw\leq \frac{\varepsilon}{5c_2c_3}\/,\quad x>x_1\/.
\end{equation}
Now, again by \eqref{eq:nu:monoton} and \eqref{eq:nu:double}, we find $c_4,c_5>0$ depending on $x_1$ such that
\begin{equation}
   \label{eq:asympt:double2}
	  \nu(x-z)\leq c_4\, \nu(x(1-\delta))\leq c_5\,\nu(x)\/,
\end{equation}
for every $z\in(0,\delta x)$ and $x>x_1$. 

The next step is to exploit \eqref{eq:ACformula} in order to write
\begin{eqnarray}
	\nonumber
   \frac{q_t^*(x)}{\nu(x)} - ({\tt d^*}+N^*(t)) &=& \int_0^t du\int_{0}^{\delta x} q_u^*(z)\left(1-\frac{z}{x}\right)
   \left(\frac{p_{t-u}(x-z)}{(t-u)\nu(x-z)}-1\right)\frac{\nu(x-z)}{\nu(x)}dz\\
	\label{eq:asympt:int12}
	&&+\int_0^t du\int_{0}^{\delta x} q_u^*(z)\left(1-\frac{z}{x}\right)\left(\frac{\nu(x-z)}{\nu(x)}-1\right)dz\\
	\label{eq:asympt:int2}
	&&+\int_0^tdu\int_{\delta x}^x \frac{q_u^*(z)}{\nu(x)}\left(1-\frac{z}{x}\right)\frac{p_{t-u}(x-z)}{(t-u)}dz\\
	&&+\int_0^t du \int_0^{\delta x}\frac{z}{x}q_u^*(z)dz-\int_0^tdu\int_{\delta x}^\infty q_u^*(z)dz+{\tt d^*}
	\left(\dfrac{p_t(x)}{t\nu(x)}-1\right)
\end{eqnarray}
and separately estimate the above-given integrals. By \eqref{eq:ptnu:asympt:xstrong} we can find $x_2>x_1$ such that 
$$
\left|\frac{p_{t-u}(x-z)}{(t-u)\nu(x-z)}-1\right|\leq \frac{\varepsilon}{5c_5(x_1)}
$$
for every $x>x_2$ and $u<t<t_0$. Consequently, for every $t<t_0$, we obtain
\begin{eqnarray*}
  \left|\int_0^t du\int_{0}^{\delta x} q_u^*(z)\left(1-\frac{z}{x}\right)\left(\frac{p_{t-u}(x-z)}{(t-u)\nu(x-z)}-1\right)
  \frac{\nu(x-z)}{\nu(x)}dz\right| &\leq& \frac{1}{5}\varepsilon \int_0^t du\int_0^{\delta x}q_u^*(z)dz\\
	&\leq& \frac{\varepsilon}{5}N^*(t)\/.
\end{eqnarray*}
Using \eqref{eq:asympt:qtu:upper}, \eqref{eq:asympt:double} and then \eqref{eq:ptint:bounds}, we estimate the positive 
expression in \eqref{eq:asympt:int2} by
\begin{eqnarray*}
	c_3\,N^*(t) \int_0^{t}du\int_{\delta x}^x \frac{\nu(z)}{\nu(x)}\left(1-\frac{z}{x}\right)\frac{p_{t-u}(x-z)}{t-u}dz
	&\leq& c_3c_2\,N^*(t) \int_0^{t_0} du\int_{0}^{(1-\delta)x} \frac{w}{x}\frac{p_{u}(w)}{u}dw\\
	&\leq& \frac{\varepsilon}{5}N^*(t)\/,
\end{eqnarray*} 
for every $t<t_0$ and $x>x_2$. Note that we used the fact that $\delta>1/2$ in the first estimate and we also applied 
the substitution $w=x-z$ in the inner integral. Integrability of the L\'evy measure gives existence of $\tilde{x}>0$ such that 
\begin{equation*}
  \int_{\tilde{x}}^\infty \nu(z)dz\leq \frac{\varepsilon}{10(t_0\vee 1)(1+c_5)c_3}\/.
\end{equation*}
Thus we can write for $x>x_3=x_2\vee (2\tilde{x})\vee (5\tilde{x}/\varepsilon)$ that
\begin{equation*}
\int_0^tdu\int_{\delta x}^\infty q_u^*(z)dz \leq c_3(t_0,x_0)t N^*(t)\int_{\delta x}^\infty \nu(z)dz\leq \frac{\varepsilon}{10}N^*(t)
\end{equation*}
and
\begin{eqnarray*}
\int_0^tdu \int_0^{\delta x}\frac{z}{x}q_u^*(z)dz &\leq& \int_0^tdu \left(\int_0^{\varepsilon x/5}+\int_{\varepsilon x/5}^{\delta x}\right)
\frac{z}{x}q_u^*(z)dz\\
&\leq & \frac{\varepsilon}{5} N^*(t)+\int_0^{t}du\int_{\varepsilon x/5}^\infty q_u^*(z)dz
\leq \left(\frac{\varepsilon}{5}+\frac{\varepsilon}{10}\right)N^*(t)\/.
\end{eqnarray*}
Collecting all together leads to
\begin{equation*}
   \left|\frac{q_t^*(x)}{\nu(x)}-N^*(t)\right|\leq \frac{4\varepsilon}{5}N^*(t)+I(t,x)+{\tt d^*}\left(\dfrac{p_t(x)}{t\nu(x)}-1\right)\/,
\end{equation*}
where
$$
I(t,x) = \int_0^t du\int_{0}^{\delta x} \left|q_u^*(z)\left(1-\frac{z}{x}\right)\left(\frac{\nu(x-z)}{\nu(x)}-1\right)\right|dz.
$$
The integrand in the integral above can be bounded in the following way
$$
 \left| \ind_{(0,\delta x)}(z)q_u^*(z)\left(1-\frac{z}{x}\right)\left(\frac{\nu(x-z)}{\nu(x)}-1\right)\right| \leq c_4\,q_u^*(z)\/,
$$
where the last function is integrable over $(0,t)\times(0,\infty)$. Thus, by dominated convergence and \eqref{eq:nu:ratio:limit}, 
we get that \eqref{eq:asympt:int12} tends to zero as $x\to\infty$, uniformly in $t<t_0$. It means that \eqref{eq:qtnu:asympt:xweak} 
follows. However, if we assume that $\nu(x)$ is non-increasing on $(x^*,\infty)$ we can deal with  integral in \eqref{eq:asympt:int12} 
more carefully in the following way. In this case we assume that ${\tt d^*}=0$, since for ${\tt d^*}>0$ the result is obvious as it was 
explained in Remark \ref{rem:dstarpositive:xinfty}. For $x>x_4 = (2x_3)\vee (2x^*)$ and $z<x_4/4$ we have $x-z\geq x_4/2>x^*$. 
Moreover $x_4/4>x_3/2>\tilde{x}$. These together with the monotonicity of $\nu(x)$ give 
\begin{eqnarray*}
   I(t,x) &=& \int_0^t du\left(\int_{0}^{x_4/4}+\int_{x_4/4}^{\delta x}\right) q_u^*(z)\left(1-\frac{z}{x}\right)\left(\frac{\nu(x-z)}
   {\nu(x)}-1\right)dz\\
	&\leq&\left(\frac{\nu(x-x_4/4)}{\nu(x)}-1\right)\int_0^t du \int_0^{x_4/4}q_u^*(z)dz+c_3c_5\,tN^*(t)\int_{x_4/4}^\infty \nu(z)dz\\
	&\leq& \left(\frac{\nu(x-x_4/4)}{\nu(x)}-1+\frac{\varepsilon}{10}\right)N^*(t)\/.
\end{eqnarray*}
Finally, using \eqref{eq:nu:ratio:limit} we can find $x_5>x_4$ such that 
$$
\frac{\nu(x-x_4/4)}{\nu(x)}-1\leq \frac{\varepsilon}{10}\/,\quad x>x_5
$$
and obtain
$$
  \left|\frac{q_t^*(x)}{\nu(x)}-N^*(t)\right|\leq \varepsilon\,N^*(t)\/,\quad x>x_5\/,\quad t<t_0
$$
which is just \eqref{eq:qtnu:asympt:xstrong}.

Now we proceed in the opposite direction and assume that \eqref{eq:qtnu:asympt:xstrong} holds. We fix $\varepsilon>0$ and we find 
$x_0>0$ such that 
\begin{equation}
\label{eq:qt:asympt:bounds2}
   \left|\frac{q_u^*(x+z)}{\nu(x+z)({\tt d^*}+N^*(u))}-1\right|\leq \frac{\varepsilon}{3(c_1\vee 1)}
\end{equation}
for every $t<t_0$ and $x>x_0$, where $c_1=c_1(x_0)>0$ is chosen by using \eqref{eq:nu:monoton} to ensure that
\begin{equation}
   \label{eq:nu:asympt:monotonic}
\nu(y)\leq c_1\nu(x)\/,\quad y\geq x\geq x_0\/.
\end{equation}
Then we use  \eqref{eq:pt:formula:exp} and \eqref{eq:n:relation} to write
\begin{eqnarray*}
   \frac{p_t(x)}{\nu(x)}-t &=& \int_0^t du \int_0^\infty \left(\frac{q_u^*(x+z)}{\nu(x+z)({\tt d^*}+N^*(u))}-1\right)({\tt d^*}+N^*(u))
   \frac{\nu(x+z)}{\nu(x)}q_{t-u}(dz)\\
	&&+\int_0^t du \int_0^\infty \left(\frac{\nu(x+z)}{\nu(x)}-1\right)({\tt d^*}+N^*(u))q_{t-u}(dz)\\
	&&+{\tt d}\left(\frac{q_t^*(x)}{\nu(x)}-N^*(t)-{\tt d^*}\right)\/. 
\end{eqnarray*}
The last term can be easily bounded directly by using \eqref{eq:qt:asympt:bounds2}
$$
{\tt d}\left(\frac{q_t^*(x)}{\nu(x)}-({\tt d^*}+N^*(t))\right)\leq \frac{\varepsilon}{3}\,{\tt d}({\tt d^*}+N^*(t))=\frac{\varepsilon}{3}\,{\tt d}N^*(t)\/,
$$ 
for $t<t_0$ and $x>x_0$. Moreover, by \eqref{eq:nu:asympt:monotonic} and \eqref{eq:qt:asympt:bounds2} we estimate 
\begin{equation*}
  \left|\int_0^t du\int_0^\infty \left(\frac{q_u^*(x+z)}{\nu(x+z)({\tt d^*}+N^*(u))}-1\right)({\tt d^*}+N^*(u))\frac{\nu(x+z)}{\nu(x)}q_{t-u}(dz)\right|
\end{equation*}
by 
\begin{equation*}	
	\frac{\varepsilon}{3}\int_0^t du\int_0^\infty({\tt d^*} +N^*(u))q_{t-u}(dz) = \frac{\varepsilon}{3}\int_0^t ({\tt d^*}+N^*(u))n(t-u<\zeta)du\/.
\end{equation*}
Note also that 
$$
 \left|\frac{\nu(x+z)}{\nu(x)}-1\right|N^*(t_0-u)\leq (c_1+1)N^*(t_0-u)
$$
Thus, by the dominated convergence, \eqref{eq:nu:ratio:limit} and monotonicity of $t\to N^*(t)$, we get
\begin{eqnarray*}
 I_2(t,x) &=& \int_0^t du\int_0^\infty \left|\frac{\nu(x+z)}{\nu(x)}-1\right|  ({\tt d^*}+N^*(u))q_{t-u}(dz)\\
 &=& \int_0^t du\int_0^\infty \left|\frac{\nu(x+z)}{\nu(x)}-1\right|  ({\tt d^*}+N^*(t-u))q_{u}(dz)\\
&\leq& \int_0^{t_0} du\int_0^\infty \left|\frac{\nu(x+z)}{\nu(x)}-1\right| ({\tt d^*}+N^*(t_0-u))q_{u}(dz) \to 0
\end{eqnarray*}
as $x\to \infty$ and the convergence is uniform in $t<t_0$.

Once again the monotonicity of $\nu(x)$ allows to estimate the integral $I_2(t,x)$ in more delicate way. First, we use \eqref{eq:qt:tail2} 
for a dual process to find $x_1>x_0 \vee 1$ such that for every $x>x_1$
$$
\int_{x}^\infty q_{t-u}(dz)\leq \frac{\varepsilon h(1)\,n(t-u<\zeta)}{3h(x)(c_1+1)}\leq \frac{\varepsilon}{3(c_1+1)}\,n(t-u<\zeta)\/.
$$
Then, by monotonicity of $\nu(x)$ on $(x^*,\infty)$ we get for $x>x_1\vee x^*$
\begin{eqnarray*}
  I_2(t,x) &=& \int_0^tdu\left(\int_0^{x_1}+\int_{x_1}^\infty\right) \left|\frac{\nu(x+z)}{\nu(x)}-1\right|({\tt d^*}+N^*(u))q_{t-u}(dz) \\
	&\leq& \int_0^{t}du({\tt d^*}+N^*(u))\left[\left(1-\frac{\nu(x+x_1)}{\nu(x)}\right)\int_0^{x_1}q_{t-u}(dz)+(c_1+1)\int_{x_1}^{\infty}
	q_{t-u}(dz)\right]\\
	&\leq& \left(1-\frac{\nu(x+x_1)}{\nu(x)}+\frac{\varepsilon}{3}\right)\int_0^t({\tt d^*} +N^*(u))n(t-u<\zeta)du\/.
\end{eqnarray*}
By \eqref{eq:nu:ratio:limit} we can find $x_2>x_1$ such that for every $x>x_2$ and $t<t_0$
$$
  I_2(t,x) \leq \frac{2\varepsilon}{3}  \int_0^t ({\tt d^*}+N^*(u))n(t-u<\zeta)du\/.
$$
Collecting all together and applying \eqref{eq:n:relation} we get
$$
\left|\frac{p_t(x)}{\nu(x)}-t\right|\leq \varepsilon t
$$
for $x>x_2$ and $t>t_0$, i.e. \eqref{eq:ptnu:asympt:xstrong} holds. This ends the proof. 

\end{proof}

\begin{proof}[Proof of Theorem \ref{thm:ftx:xinfty}]
 By \eqref{eq:ptnu:asympt:strong2} and Theorem \ref{thm:asympt:x} we get that \eqref{eq:qtnu:asympt:xweak} holds, i.e. for given $\varepsilon>0$ we can find $x_0>0$ such that
$$
   \left|\frac{q_t^*(x)}{\nu(x)}-({\tt d^*}+N^*(t))\right|\leq \varepsilon\/,\quad x>x_0\/,\quad t<t_0\/.
$$
Thus, using \eqref{eq:ft:formula} and \eqref{eq:n:relation}, we get
\begin{equation*}
   \frac{f_t(x)}{\nu(x)}-t = \int_0^t n(t-u<\zeta)\left(\frac{q_u^*(x)}{\nu(x)}-({\tt d^*}+N^*(u))\right)du+{\tt d}\left(\frac{q_t^*(x)}
   {\nu(x)}-({\tt d^*}+N^*(t))\right)\/,
\end{equation*}
which gives that for every $t<t_0$ and $x>x_0$ we obtain
\begin{equation*}
 \left|\frac{f_t(x)}{\nu(x)}-t\right|\leq \varepsilon\left(N(t_0)+{\tt d}\right)
\end{equation*}
and the result follows. If additionally $\nu(x)$ is non-increasing for large $x$, we use Theorem \ref{thm:asympt:x} to show \eqref{eq:qtnu:asympt:xstrong}. Since $f_t(x)/\nu(x)-t$ can be written as a sum of 
\begin{equation*}
\int_0^t n(t-u<\zeta)\left(\frac{q_u^*(x)}{({\tt d^*}+N^*(u))\nu(x)}-1\right)({\tt d^*}+N^*(u))du
\end{equation*}
and
\begin{equation*}
{\tt d} ({\tt d^*}+N^*(t))\left(\frac{q_t^*(x)}{({\tt d^*}+N^*(t))\nu(x)}-1\right)\/.
\end{equation*}
 Consequently, we get
\begin{eqnarray*}
 \left|\frac{f_t(x)}{\nu(x)}-t\right|  &\leq& \varepsilon \left(\int_0^t n(t-u<\zeta)N^*(u)du+{\tt d}N^*(t)\right)=\varepsilon t\/,
\end{eqnarray*} 
for given $\varepsilon>0$, every $t<t_0$ and $x>x_0$, where $x_0$ is large enough. This ends the proof.
\end{proof}

\subsection*{Acknowledgments.} We are grateful to Tomasz Grzywny for valuable discussions and pointing out the examples considered in Section 4.2.

\end{document}